\documentclass[12pt,a4paper, oneside, reqno]{amsart}

\usepackage[foot]{amsaddr} 

\usepackage{amsmath,amsthm,amssymb,bbm}
\usepackage[bookmarksnumbered=false, colorlinks=true,citecolor=red,urlcolor=green,linkcolor=blue]{hyperref}
\usepackage{caption}
\usepackage{enumitem}
\usepackage{csquotes}
\usepackage{textgreek,textcomp}
\usepackage{graphicx,float}
\usepackage{physics}

\usepackage{dutchcal} 

\usepackage[a4paper,margin=2.5cm]{geometry}

\usepackage[labelformat=simple]{subcaption}

\theoremstyle{plain}
\newtheorem{theorem}{Theorem}[section]

\newtheorem{lemma}[theorem]{Lemma}
\newtheorem{proposition}[theorem]{Proposition}
\theoremstyle{definition}
\newtheorem{remark}[theorem]{Remark}

\newcommand{\D}[1]{\mathop{\mathrm{d}#1}}
\DeclareMathOperator{\Arg}{Arg}
\DeclareMathOperator*{\Real}{Re}
\DeclareMathOperator*{\conv}{conv}

\newcommand{\calP}{\mathsf{P}}
\newcommand{\calA}{\mathsf{A}}
\newcommand{\calD}{\mathsf{D}}
\newcommand{\calR}{\mathsf{R}}
\newcommand{\calr}{\mathsf{r}}
\newcommand{\calH}{\mathrm{H}}
\newcommand{\fatW}{\mathbf{W}}
\newcommand{\bbE}{\mathbb{E}}
\title[Size of the convex hull of planar Brownian motion]{Bounds on the size of the convex hull of planar Brownian motion and related inverse processes}

\author[W.\ Cygan]{Wojciech Cygan $^{1,2}$}
\address{$^{1}$University of Wroc\l{}aw,
		Faculty of Mathematics and Computer Science\\
		Institute of Mathematics,
		pl.\ Grunwaldzki 2, 50--384 Wroc\l{}aw, Poland}
\address{$^{2}$Technische Universit\"{a}t Dresden,
		Faculty of Mathematics\\
		Institute of Mathematical Stochastics,
		Zellescher Weg 25, 01069 Dresden, Germany}
\email{wojciech.cygan@uwr.edu.pl
}

\author[H.\ Panzo]{Hugo Panzo $^{3}$}
\address{$^{3}$Department of Mathematics and Statistics, Saint Louis University, St.\ Louis, USA}
\email{hugo.panzo@slu.edu}

\author[S.\ \v{S}ebek]{Stjepan\ \v{S}ebek $^{4}$}
\address{$^{4}$Department of Applied Mathematics\\
	Faculty of Electrical Engineering and Computing\\
	University of Zagreb\\ 
 Zagreb\\ 
	Croatia}
\email{stjepan.sebek@fer.unizg.hr}

\subjclass[2010]{
Primary 
60D05
, 60J65
; Secondary 52A10
.}
\keywords{Brownian motion, Cauchy's surface area formula, convex hull, radial slit plane, inradius, circumradius}

\begin{document}

\begin{abstract}
We establish bounds on expected values of various geometric quantities that describe the size of the convex hull spanned by a path of the standard planar Brownian motion. Expected values of the perimeter and the area of the Brownian convex hull are known explicitly, and satisfactory bounds on the expected value of its diameter can be found in the literature as well. In this work we investigate circumradius and inradius of the Brownian convex hull and obtain lower and upper bounds on their expected values. Our other goal is to find bounds on the related inverse processes (that correspond to the perimeter, area, diameter, circumradius and inradius of the convex hull) which provide us with some information on the speed of growth of the size of the Brownian convex hull.
\end{abstract}

\maketitle


\section{Introduction and Main Results}

Let $\fatW = \{\fatW(t) : t\ge 0\}$ denote the standard planar Brownian motion started at the origin, where $\mathbf{W}(t)= (W_1(t), W_2(t))$. In other words, the two coordinates $W_i(t)$, $i=1,2$, are independent, one-dimensional standard Brownian motions started at $0$. For a set $A\subset \mathbb{R}^2$, the set $\conv A$ is the convex hull spanned by $A$, i.e.,\ the smallest convex subset of $\mathbb{R}^2$ containing $A$. Let $\calH(t) = \conv \{\mathbf{W}(s): s\in [0,t]\}$ be the convex hull generated by a path of $\fatW$ run up to time $t$.  Questions concerning the shape  of the convex hull of planar Brownian motion and fine properties of its boundary have been investigated for many years and the first conjecture was formulated already by P.\ L\'{e}vy in \cite{Levy} where he proposed that the boundary of $\calH(t)$ should be of class $C^1$. This conjecture turned out to be true and it was proved in \cite{ElBachir}; see also \cite{Cranston} for another proof.

Convex hulls of the range of random walks is another strongly related and very active field of current research. In the present context, most relevant are the works concerned with computations of the expected value of various geometric functionals that describe the shape of the convex hull. In the seminal paper \cite{Spitzer-Widom} the authors found a very useful formula (through an application of the Cauchy formula) which allowed one to compute the expected perimeter (see also \cite{Baxter} for another (combinatorial) proof). Later, a number of works concentrated on several strongly related questions; see e.g.\ \cite{Snyder-Steele, Majumdar, Wade-Xu-SPA-2015, Vysotsky-Zapor, Cygan-Sandric-Sebek, ch_bm_bb}. We refer the reader to check references therein for further results.

In higher dimensions, instead of the functionals studied here, one usually investigates the so-called intrinsic volumes; see \cite{Kabluchko-Zapor-TAMS, Eldan} for the case of the multidimensional Brownian convex hull, and \cite{Molchanov-Wespi, Molchanov} for the case of convex hulls spanned by L\'{e}vy processes. It is worth mentioning, however, that the same functionals we study in this paper can also be treated in higher dimensions; see \cite{high_dim_hulls}.
 
The main interest of the present article is to find bounds on various geometric quantities that describe the size of the convex hull $\calH(t)$, such as perimeter, area and diameter. Since the expected values of  the perimeter and the area of $\calH(t)$ are explicitly known, and bounds on the expected diameter of $\calH(t)$ can be found in the literature  as well, we focus on establishing bounds for related inverse processes. This gives us information on how much time it takes for planar Brownian motion to span a convex hull of a given size that is measured in terms of these geometric quantities. In addition to these basic geometric functionals, we also analyze the circumradius and inradius of $\calH(t)$ and their corresponding inverse processes. As far as the authors know, these functionals of the convex hull of planar Brownian motion have yet to be studied in the literature. 

We start by recalling the known results for the expected value of the perimeter and the area of $\calH(t)$, and for bounds on the expected value of the diameter of $\calH(t)$. Next, we present our bounds on the expected circumradius and the expected inradius of $\calH(t)$, and later we present our estimates for the related inverse processes. Let $\calP(t)$ denote the perimeter, $\calA(t)$ the area, and $\calD(t)$ the diameter of $\calH(t)$. Throughout we (usually) abandon the time index whenever $t=1$, that is we write $\calP $ for $\calP(1)$ and similarly for the other quantities. 
The expected value of $\calP(t)$ was computed in \cite{Letac} (see also \cite{Majumdar}) and it is given by
\begin{equation}\label{eq:perimeter_moment}
\mathbb{E}[\calP(t)]=\sqrt{8\pi t}.
\end{equation}
The proof of \eqref{eq:perimeter_moment} in \cite{Letac} combines the reflection principle for Brownian motion with the classical Cauchy's surface area formula, see e.g.\ \cite[Theorem 1]{Cauchy}, 
\begin{equation}\label{Cauchy-Per}
\calP(t) = \frac{1}{2}\int_{-\pi}^\pi\left(\sup_{0\leq s \leq t}\fatW(s)\cdot \mathbf{e}_\theta - \inf_{0\leq s\leq t}\fatW(s) \cdot \mathbf{e}_\theta\right)\D{\theta}.
\end{equation}
Here we write $\mathbf{e}_\theta$ for the vector $(\cos \theta, \sin \theta)$. Using \eqref{Cauchy-Per} and a calculation from \cite{Rogers_Shepp}, the authors of \cite{Wade_Xu} were able to give the following integral expression for the second moment of the perimeter of the hull $\calH$ (which is spanned by a path run up to time 1):
\begin{equation}\label{eq:perimeter_2nd}
\mathbb{E}\left[\calP^2\right]=4\pi\int_{-\frac{\pi}{2}}^\frac{\pi}{2}\int_0^\infty \cos\theta\frac{\cosh(u\theta)}{\sinh(u\pi/2)}\tanh\left(\frac{2\theta+\pi}{4}u\right)\D{u}\D{\theta}
\approx 26.209056.
\end{equation}
\begin{remark}
	We remark that the numerical value given in \eqref{eq:perimeter_2nd} differs from the one reported in \cite{Wade_Xu}. The reason is that there appear to be some numerical issues concerning the evaluation of the integral from \eqref{eq:perimeter_2nd}	if one tries to compute it directly through the displayed formula. However, these inaccuracies can be removed by performing some simple symbolic manipulations; see Subsection \ref{subsec:double_integral} for more details.
\end{remark}

The expected value of the area of $\calH(t)$ was found in \cite{ElBachir} (see also \cite{Majumdar}) and equals
\begin{equation}\label{eq:area_moment}
\mathbb{E}[\calA(t)]=\frac{\pi}{2}t.
\end{equation}
There is no known closed-form formula for the expectation of the diameter of $\calH(t)$ (which is evidently equal to the diameter of the trajectory run up to time t), but various bounds have appeared in the literature. The most simple estimate follows from convexity, namely the almost sure inequalities $2 \le \calP / \calD \le \pi$ hold, where the two extrema are realized by a line segment and curves of constant width. Combining this with \eqref{eq:perimeter_moment} immediately implies
\begin{equation*}
	1.5957 \approx \sqrt{\frac{8}{\pi}} \le \mathbb{E}[\calD] \le \sqrt{2\pi} \approx 2.5067.
\end{equation*}
In \cite{McRedmond_Xu} the authors improved upon these estimates and they showed that 
\begin{equation}\label{bound-McRedmond-Xu}
1.6014 \leq \mathbb{E}[\calD] \leq \sqrt{8\log 2} \approx 2.3549.
\end{equation}
The main idea behind the proof of \eqref{bound-McRedmond-Xu} was the following observation
\begin{equation}\label{Range-bound-Diam}
\max \{R_1,R_2\}\leq \calD \leq \sqrt{R_1^2+R_2^2},
\end{equation}
where 
\begin{equation}\label{Range-process}
R_i(t) = \sup_{0\leq s \leq t}W_i(s) - \inf_{0\leq s\leq t}W_i(s),\quad i=1,2
\end{equation}
are two independent  \textit{range} processes of the two coordinates of $\fatW$. For clarity we usually write $R_i$ instead of $R_i(1)$. The lower bound in \eqref{Range-bound-Diam} is obvious and the upper bound follows from the fact that $\calH$ is contained in the rectangle of width $R_1$ and height $R_2$, which forces the diameter $\calD$ to be smaller than the diagonal of the rectangle. To obtain the numerical upper bound from \eqref{bound-McRedmond-Xu}, the authors used the expression for the second moment of $R_1$ from \cite{Feller}. In \cite{Jovalekic} the lower bound from \eqref{bound-McRedmond-Xu} has been improved to
\begin{equation}\label{Jovalekic}
 \mathbb{E}[\calD] \geq 1.856.
\end{equation}
To obtain this bound, the author found a series representation for the distribution function of the random variable $\max\{R_1,R_2\}$ which is based on the well-known (see \cite{Feller}) formula for the density of $R_i$ that takes the form 
\begin{equation}\label{eq:range_density}
f_R(x) = \frac{8}{\sqrt{2\pi}} \sum_{n=1}^\infty (-1)^{n-1}n^2 e^{-\frac{1}{2}n^2x^2},\quad x>0.
\end{equation}
At this point we should also mention an unpublished work \cite{Garbit-Raschel}  (it was kindly presented to us by the authors) where essentially the same bound as in \eqref{Jovalekic} was found with similar methods as in \cite{Jovalekic}.
We also refer to \cite{Vysotsky} for Large Deviation Principles for the perimeter and the area of $\calH(t)$.

Before presenting the statements of our main results, we display a table with a summary on the obtained bounds, together with Monte Carlo estimates of the studied quantities. We denote the circumradius of $\calH(t)$ by $\calR(t)$, and the inradius by $\calr(t)$. Moreover, for a one-dimensional nondecreasing stochastic process $X(t)$, we use the following notation for the related \textit{inverse} process
\[
\Theta^X(y) = \inf\{t\geq 0:X(t)>y\},\quad y\geq 0.
\]
Remember that we usually use the convention of abandoning the time indices $t$ or $y$ whenever they are equal to one.
\begin{table}[!h]
\begin{center}
\renewcommand{\arraystretch}{1.3}
	\begin{tabular}{ |c|c|c|c|c| } 
		\hline
		Quantity & Lower bound & Estimated mean & Upper bound & Corresponding result \\ 
		\hline \hline
		$\mathbb{E}[\calR]$& $0.928$ & $0.999$ & $1.1775$ & Proposition \ref{prop:circumradius_bounds} \\
		\hline
		$\mathbb{E}[\calr]$& $0.3930$ & $0.513$ & $0.7072$ & Proposition \ref{prop:inradius_bounds} \\
		\hline
		$\mathbb{E}[\Theta^{\calP}]$ & $0.0397$ & $0.045$ & $0.0507$ & Proposition \ref{prop:inverse-per} \\
		\hline
		$\mathbb{E}[\Theta^{\calA}]$ & $0.6366$ & $0.718$ & $1.6651$ & Proposition \ref{prop:inverse-area}\\
		\hline
		$\mathbb{E}[\Theta^{\calD}]$ & $0.1803$ & $0.303$ & $0.3466$ & Proposition \ref{prop:inverse-diam} \\
		\hline
		$\mathbb{E}[\Theta^{\calR}]$ & $0.7213$ & $1.192$ & $1.3863$ & Proposition \ref{prop:inverse-circumradius} \\
		\hline
		$\mathbb{E}[\Theta^{\calr}]$ & $2$ & $4.176$ & $83.4$ & Proposition \ref{prop:inverse-inrad} \\
		\hline
	\end{tabular}
	\renewcommand{\arraystretch}{1}
	\vspace{0.3cm}
	\caption{Our bounds on the mean compared with estimated values coming from the Monte Carlo simulations.}
	\label{Table-simulations}
\end{center}
\end{table}

\subsection*{Bounds on the expectation of circumradius and inradius}
In addition to perimeter, area and diameter, other intriguing geometric quantities which give valuable information on the size of the set $\calH(t)$ include its circumradius $\calR(t)$, and its inradius $\calr(t)$. The circumradius is the radius of the smallest circle which contains $\calH(t)$ (which is the same thing as the circumradius of the trajectory up to time t), and the inradius is the radius of the largest circle which is contained in $\calH(t)$; see Figure \ref{fig:Incircle}. We recall that the center of the largest circle contained in a set is called Chebyshev's center. Notice that Chebyshev's center does not have to be unique, but the inradius is still well defined. Chebyshev's center plays an important role in convex optimization problems; see \cite{Boyd}. We will come back to this issue in Section \ref{sec:inrad}.
\begin{figure}[h]
\centering
\includegraphics[scale=0.5]{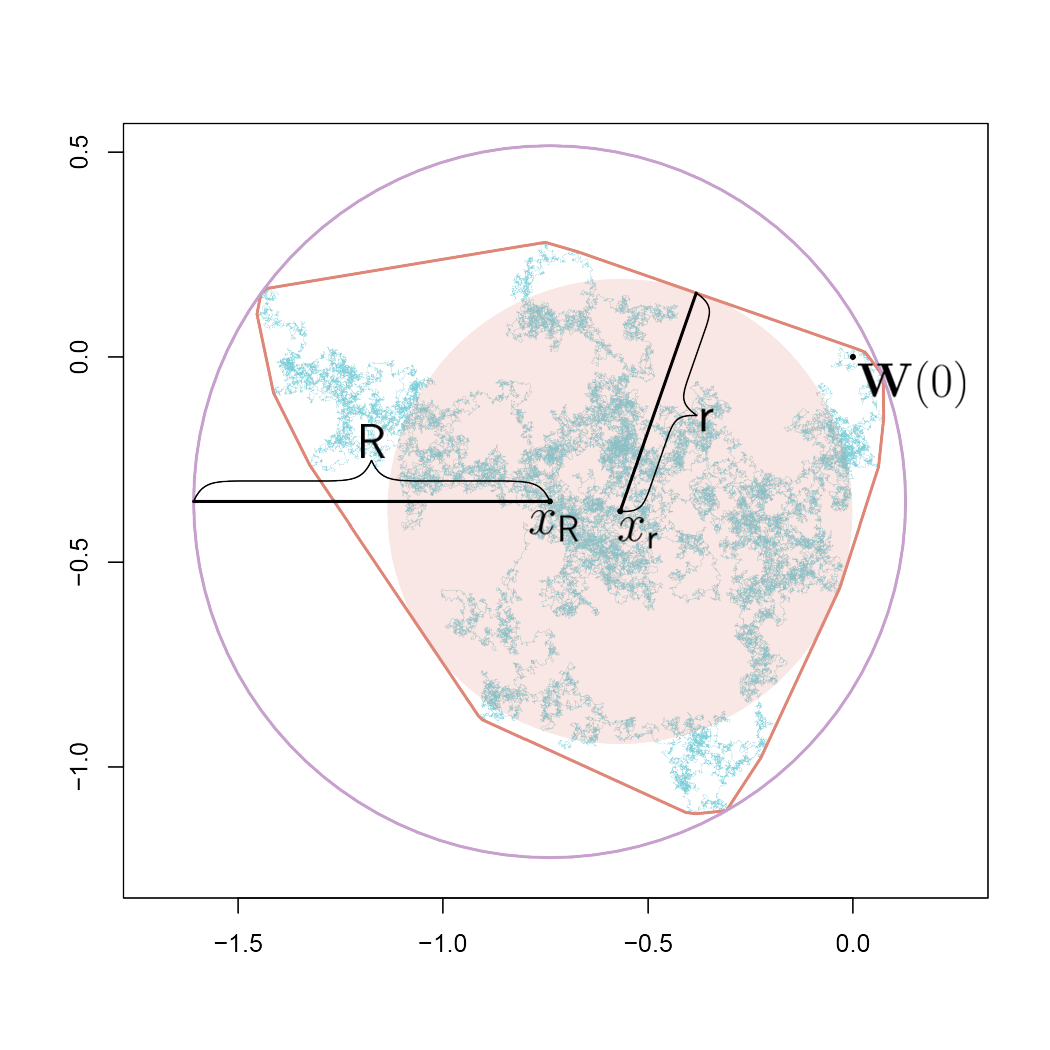}
\vspace{-0.5cm}
\caption{The convex hull $\calH$ together with the smallest circle in which it is contained, and the largest circle contained in $\calH$. The center of the smallest circle which contains $\calH$ is denoted by $x_{\calR}$, and the corresponding circumradius is $\calR$. The center of the largest circle contained in $\calH$ (Chebyshev's center) is denoted by $x_{\calr}$, and the corresponding inradius is $\calr$.}
\label{fig:Incircle}
\end{figure} 

To the best of our knowledge, no bounds on the expected value of $\calR(t)$ and $\calr(t)$  are available in the literature, so one of the contributions of this article is to establish  lower and upper bounds for $\mathbb{E}[\calR(t)]$ and $\mathbb{E}[\calr(t)]$. These bounds are presented in the following propositions, cf.\ Table \ref{Table-simulations} for Monte Carlo estimates of $\mathbb{E}[\calR]$ and $\mathbb{E}[\calr]$. We remark that in light of Brownian scaling, it is enough to obtain bounds only for $t=1$, cf.\ Proposition \ref{prop:equivalence}.

We first observe that we can obtain a lower and an upper bound on the expected value of the circumradius directly from the trivial almost sure inequality $4\calR \le \calP \le 2\pi\calR$ (the extrema being realized by a line segment and a disk), and formula \eqref{eq:perimeter_moment}. More precisely, this yields
\begin{equation*}
	0.7978 \approx \sqrt{\frac{2}{\pi}} \le \mathbb{E}[\calR] \le \sqrt{\frac{\pi}{2}} \approx 1.2534.
\end{equation*}
We next improve on both of these bounds.
\begin{proposition}\label{prop:circumradius_bounds}
In the above notation it holds that
\begin{equation}\label{bounds-circumradius}
0.928 \leq \mathbb{E} [\calR]\leq \sqrt{2 \log 2} \approx 1.1775.
\end{equation}
\end{proposition}
To prove the lower bound in \eqref{bounds-circumradius} we use the fact that $\calD\leq 2 \calR$ and combine it with \eqref{Jovalekic}. For the upper bound in \eqref{bounds-circumradius} we use (similarly as it was used in \cite{McRedmond_Xu} to prove \eqref{Range-bound-Diam}) the fact that $\calH$ is contained in the rectangle of width $R_1$ and height $R_2$. For a detailed proof see Section \ref{sec:inradius}.

In the following proposition, we find an alternative upper bound for $\mathbb{E}[\calR]$ by considering the smallest ball centered at the origin which contains $\calH$. For details see Section \ref{sec:inradius}.

\begin{proposition}\label{prop:alternative-upper-circumnradius}
	In the above notation it holds that
	\[
	\mathbb{E}[\calR]\leq\Gamma(1/2)^{-1}\int_0^\infty \frac{t^{-1/2}}{I_0(\sqrt{2t)}}\D{t}\approx 1.662180,
	\]
where $I_\nu(x)$ denotes the modified Bessel function of the first kind of order $\nu$.	
\end{proposition}	

When it comes to inradius, we obtain the following bounds.
\begin{proposition}\label{prop:inradius_bounds}
In the above notation it holds that
\begin{equation}\label{bounds-Inradius}
0.3930\approx\frac{1}{2\pi}\left(\sqrt{8\pi}-\sqrt{\mathbb{E}\left[\calP^2\right]-2\pi^2}\right)\leq\mathbb{E}[\calr]\leq \frac{\sqrt{2}}{2}\approx 0.7072 .
\end{equation}
\end{proposition}
To prove the lower bound in \eqref{bounds-Inradius} we apply a Bonnesen-type inequality from \cite{Bonnesen} that involves inradius, perimeter, and area (see \cite[Equation 15]{Bonnesen}), and we combine it with Jensen's inequality and the known results from \eqref{eq:perimeter_moment}, \eqref{eq:perimeter_2nd} and \eqref{eq:area_moment}.
For the upper bound in \eqref{bounds-Inradius} we use the obvious fact that $\calA\geq \pi \calr^2$ and combine it with  Jensen's inequality. For a detailed proof see Section \ref{sec:inradius}.

In the following proposition, we find an alternative lower bound for $\mathbb{E}[\calr]$ by considering a particular triangle, embedded in the path of the Brownian motion, whose inradius is amenable to estimation. While this lower bound is somewhat worse than the one in \eqref{bounds-Inradius}, it has the advantage of being constructive. For details see Section \ref{sec:inradius}.

\begin{proposition}\label{prop:alternative-lower-Inradius}
In the above notation it holds that
\[
\mathbb{E}[\calr]\geq\frac{\sqrt{\pi}}{8+4\sqrt{2}}\approx 0.1297 .
\]
\end{proposition}

\subsection*{Bounds on expected times to reach a given size}
There is a question which arises naturally in the present context, namely \textit{how much time does it take for planar Brownian motion to span a convex hull of a given size}? This question can be investigated in terms of the basic geometric quantities such as perimeter, area, diameter, circumradius and inradius, and it results in establishing bounds on the expectation of the corresponding inverse processes. The main aim of this article is to find such bounds and thus to give (at least a partial) answer to the question posed above.

We consider the inverse processes $\Theta^\calP(y)$, $\Theta^\calA(y)$, $\Theta^\calD(y)$, $\Theta^\calR(y)$ and $\Theta^\calr(y)$ that correspond to the processes $\calP(t)$, $\calA(t)$, $\calD(t)$, $\calR(t)$ and $\calr(t)$, respectively. We start by presenting bounds on the inverse perimeter process. We remark that in view of the scaling properties of the studied inverse processes (see Proposition \ref{prop:equivalence}), it is enough to give bounds only in the case when $y=1$. 

\begin{proposition}\label{prop:inverse-per}
In the above notation it holds that
\begin{equation}\label{eq:inverse-bounds-per}
0.0397\approx\frac{1}{8\pi}\leq\mathbb{E}\left[\Theta^\calP\right]\leq \frac{1}{2\pi^2}\approx 0.0507.
\end{equation}
\end{proposition}
The lower bound in \eqref{eq:inverse-bounds-per} follows from the corresponding equivalence in law from Proposition \ref{prop:equivalence} combined with Jensen's inequality and \eqref{eq:perimeter_moment}. The upper bound also rests on Jensen's inequality, but this time it is  combined with Cauchy's surface area formula \eqref{Cauchy-Per}. For more details we refer to Section \ref{sec:Inverse-bounds}.

Before presenting bounds on the expectation of the inverse area process, we first need to state a result concerning the expected value of the minimum of two independent range processes at time $1$. This quantity, which we denote by the symbol $\mathfrak{r}$, features in several of our estimates and we derive an integral expression for it in the following lemma. Despite its relatively simple form, the definite integral seems difficult to evaluate explicitly, though it can be evaluated numerically to arbitrary precision using software such as \emph{Mathematica}. 

\begin{lemma}\label{lem:min_range}
It holds that
\begin{equation}\label{eq:frak_constant}
\begin{split}
\mathfrak{r}:=\mathbb{E}\big[\min\big\{\Theta^{R_1},\, \Theta^{R_2}\big\}\big]
&=\frac{1}{2}-\frac{2}{\pi}\int_0^\infty \left(1-\frac{4}{(\cos t+\cosh t)^2}\right)\frac{1}{t^3}\D{t}\\
&\approx 0.346554.
\end{split}
\end{equation}
\end{lemma}

For the proof of Lemma \ref{lem:min_range}, we find a closed-form formula for the characteristic function of the random variable $\Theta^{R_1} - \Theta^{R_2}$ which is valid on the whole positive real semi-axis; see Section \ref{sec:Range} for details.

\begin{proposition}\label{prop:inverse-area}
In the above notation it holds that
\begin{equation}\label{eq:inverse-bounds-area}
0.6366\approx\frac{2}{\pi}\leq\mathbb{E}\left[\Theta^\calA\right]
\leq 2\sqrt{2 \mathfrak{r}}\approx 1.6651,
\end{equation}
where $\mathfrak{r}$ is the constant defined in \eqref{eq:frak_constant}.
\end{proposition}

 The lower bound in \eqref{eq:inverse-bounds-area} follows from the corresponding equivalence in law from Proposition \ref{prop:equivalence} combined with Jensen's inequality and \eqref{eq:area_moment}. 
To find the upper bound in \eqref{eq:inverse-bounds-area}, we perform a two-stage geometric procedure which results in squeezing two adjacent triangles into $\calH(t)$, and enables us to bound the expected value of $\Theta^{\calA}$ in terms of $\mathbb{E}\big[\min\big\{\Theta^{R_1},\, \Theta^{R_2}\big\}\big]$. For details we refer to Section \ref{sec:Inverse-bounds}.

Bounds on the expectation of the inverse diameter process are displayed below. 

\begin{proposition}\label{prop:inverse-diam}
In the above notation it holds that
\begin{equation}\label{eq:inverse-bounds-diam}
0.1803\approx\frac{1}{8\log 2}\leq\mathbb{E}\left[\Theta^{\calD}\right]\leq \mathfrak{r}\approx 0.3466,
\end{equation}
where $\mathfrak{r}$ is the constant defined in \eqref{eq:frak_constant}.
\end{proposition}

The lower bound in \eqref{eq:inverse-bounds-diam} follows from the corresponding equivalence in law from Proposition \ref{prop:equivalence} combined with Jensen's inequality and the upper bound from \eqref{bound-McRedmond-Xu}. The upper bound in \eqref{eq:inverse-bounds-diam} follows from the obvious fact that as soon as one of the two coordinates of $\fatW$ attains range $1$, then the diameter of $\calH(t)$ must be at least $1$. For more details we refer to Section \ref{sec:Inverse-bounds}.

Bounds on the expectation of the inverse circumradius process are presented below.

\begin{proposition}\label{prop:inverse-circumradius}
In the above notation it holds that 
\begin{equation}\label{eq:inverse-bounds-circumradius}
 0.7213 \approx \frac{1}{2 \log 2} \leq \mathbb{E}\left[\Theta^\calR\right]\leq 4\mathfrak{r}\approx 1.3863,
\end{equation}
where $\mathfrak{r}$ is the constant defined in \eqref{eq:frak_constant}.
\end{proposition}
The lower bound in \eqref{eq:inverse-bounds-circumradius} follows from the corresponding equivalence in law from Proposition \ref{prop:equivalence} combined with Jensen's inequality and the upper bound from \eqref{bounds-circumradius}. For the upper bound, notice that as soon as one of the two independent range processes hits value $2$, the circumradius must be at least $1$. For a detailed proof see Section \ref{sec:Inverse-bounds}.

The most challenging task of this article is to obtain bounds (especially an upper bound) on the expected time needed for the inradius process to attain value $1$.

\begin{proposition}\label{prop:inverse-inrad}
In the above notation it holds that
\begin{equation}\label{eq:inverse-bounds-inrad}
2\leq\mathbb{E}\left[\Theta^\calr\right]\leq 83.4 .
\end{equation}
\end{proposition}

While the upper bound in \eqref{eq:inverse-bounds-inrad} seems far from being sharp, cf.\ Table \ref{Table-simulations}, we find the method that allowed us to obtain this upper bound very interesting nevertheless.
We now embark on a brief discussion of our approach.
To prove the lower bound in \eqref{eq:inverse-bounds-inrad}, we apply the isoperimetric inequality combined with the corresponding equivalence in law from Proposition \ref{prop:equivalence} (v) and the lower bound from \eqref{eq:inverse-bounds-area}.
For the upper bound, we employ a two-stage construction similar to that of Proposition \ref{prop:inverse-area}. More precisely, we first wait until one of the coordinates of $\fatW$ attains a given range. This waiting time is distributed like the random variable $\min\{\Theta^{R_1}(a),\, \Theta^{R_2}(a)\}$ for a parameter $a>0$. At this time the convex hull will contain a line segment of length at least $a$ with the Brownian motion being located at one of the endpoints of this line segment. 
To explain the second stage of our construction, we recall the definition of a \textit{radial slit plane}. For $n\in\mathbb{N}$, we denote by $\mathcal{S}_n$ the complex plane with $n$ equally spaced radial slits emanating from the unit disk. More precisely, 
\begin{equation}\label{eq:slit_plane}
\mathcal{S}_n=\mathbb{C}\setminus\{z\in\mathbb{C}:|z|\geq 1~\mathrm{ and }\,\Arg z^n=0\}.
\end{equation}
In the second stage of the construction, we wait until $\fatW$ exits a radial slit plane $r\mathcal{S}_6$ (i.e.\ with $6$ slits that reach to within distance $r$ of the origin), where $r>0$ is another parameter. 
After this step is completed, the inradius of $\calH$ will be bigger than some explicit function $\varrho (a,r)>0$, which is equal to the inradius of a certain triangle that is contained in the convex hull. This will result in the following bound
\begin{equation}\label{eq:min-range+slit}
\mathbb{E}\left[\Theta^{\calr}(\varrho(a,r))\right]\leq \mathbb{E}\left[\min\big\{\Theta^{R_1}(a),\, \Theta^{R_2}(a)\big\}\right]+\mathbb{E}\left[\tau_{r\mathcal{S}_6}\right],
\end{equation}
where $\tau_{r\mathcal{S}_n}$ stands for the exit time of $\fatW$ from $r\mathcal{S}_n$. Successfully exploiting estimate \eqref{eq:min-range+slit} requires the precise value of $\mathbb{E}[\tau_{\mathcal{S}_n}]$, and this result, which is itself a new contribution, is displayed in the following theorem (its proof is given in Section \ref{sec:slit}). Since each ``petal'' of the radial slit plane is essentially a wedge with opening angle $2\pi/n$, an intuitive explanation for the integrability transition at $n=4$ can be had by appealing to Spitzer's well-known integrability condition for the exit time of wedges; see \cite[Theorem 2]{Spitzer}.

\begin{theorem}\label{lem:slit_exit}
In the above notation it holds that
\[
\mathbb{E}\left[\tau_{\mathcal{S}_n}\right]=\begin{cases}
\infty&n\leq 4\\
\displaystyle\frac{\Gamma\left(\frac{1}{2}-\frac{2}{n}\right)}{2\sqrt{\pi}\,\Gamma\left(1-\frac{2}{n}\right)}&n> 4.
\end{cases}
\]
\end{theorem}

\noindent Finally, estimate \eqref{eq:min-range+slit} can be combined with Theorem \ref{lem:slit_exit} and Lemma \ref{lem:min_range}, and accompanied by an optimization procedure along parameters $a$ and $r$ to yield the desired upper bound in \eqref{eq:inverse-bounds-inrad}. For a detailed proof see Section \ref{sec:inverse-inradius-bounds}.

\begin{remark}
In the two-stage constructions used to prove Propositions \ref{prop:inverse-area} and \ref{prop:inverse-inrad}, we could have alternatively constructed the first line segment of length at least $a$ by waiting until $\fatW$ exits from a ball of radius $a$. However, this turns out to result in a worse estimate since the expected time until $\fatW$ exits from a ball of radius $a$ (starting from the center of the ball) is $a^2/2$, which is bigger than the expected value of the random variable $\min\{\Theta^{R_1}(a),\, \Theta^{R_2}(a)\}$; see \eqref{eq:range_equivalence} and Lemma \ref{lem:min_range}.
\end{remark}

The rest of the article is organized as follows. In Section \ref{sec:Range} we study the minimum of two independent inverse range processes at time $1$. In Section \ref{sec:inradius} we present bounds on the expectation of the circumradius and the  inradius of $\calH$. In Section \ref{sec:inverse-processes-all} we give proofs of bounds for all inverse processes. Finally, in Section \ref{sec:simulations} we discuss some technical details related to simulations and numerical integration.

\section{Minimum of two Inverse Range processes}\label{sec:Range}

This section is devoted to the proof of Lemma \ref{lem:min_range}. Before we embark on the proof, we start with some preparation which concerns a one-dimensional range process related to a Brownian motion. 
For this reason we set $W(t)$ to be a one-dimensional Brownian motion started at $0$ and let $R(t)$ be its range process defined as
\begin{align*}
R(t) = \sup_{0\leq s \leq t}W(s) - \inf_{0\leq s\leq t}W(s),\quad t\geq 0,
\end{align*}
cf.\ \eqref{Range-process}.
We remark that the convex hull spanned by a path of $W(t)$ up to time $t$ is nothing but the segment $[\inf_{0\leq s\leq t}W(s), \sup_{0\leq s\leq t}W(s)]$ whose length is equal to $R(t)$. The right-continuous inverse range process $\{\Theta(y):\, y\geq 0\}$
is defined as the exit time
\[
\Theta(y)=\inf\{t\geq 0:R(t)>y\}.
\]
It is known that $\Theta(y)$ has independent increments. 
The processes $R(t)$ and $\Theta(y)$ are both well studied and their marginal laws are known explicitly.
By \cite[Equation 3.6]{Feller}, the density of $R(1)$ is given by equation \eqref{eq:range_density}.
Before we present a formula for the density of the random variable $\Theta(1)$, we first recall that its Laplace transform is given by, see \cite{Imhof}, 
\begin{equation}\label{eq:Lap_transform}
\mathbb{E}\left[e^{-\lambda \Theta(y)}\right]=\sech^2\left(y\sqrt{\lambda/2}\right),\quad \lambda\geq 0.
\end{equation}
Further, the authors of  \cite{theta_laws} studied the following random variable
\begin{align*}
C_2 = \frac{2}{\pi^2}\sum_{n=1}^\infty \frac{\Gamma_{2,n}}{(n-\frac{1}{2})^2},
\end{align*}
where $\Gamma_{2,n}$ are independent random variables with distribution $\mathrm{Gamma}(2,1)$.
According to \cite[Eq.\ (1.8)]{theta_laws}, the random variable $C_2$ has the Laplace transform
\begin{align}\label{Lap-C_2}
\mathbb{E}\left[  e^{-\lambda C_2}\right] = \frac{1}{\cosh^2 \big( \sqrt{2\lambda }\big)}, \quad \lambda \geq 0.
\end{align}
In view of the hyperbolic identity $\sech x =1/\cosh x$, and by comparing \eqref{eq:Lap_transform} and \eqref{Lap-C_2}, we conclude that $\Theta(1)\stackrel{d}{=} 4^{-1}C_2$.
Thus, we can find the density of $\Theta(1)$ in \cite[Table 1]{theta_laws} and it is given by
\begin{equation}\label{eq:rt_density}
f_\Theta (t)=4\sum_{n=0}^\infty\left((2n+1)^2\pi^2 t-1\right)\exp\left(-\frac{(2n+1)^2\pi^2 }{2}t\right),\quad t>0.
\end{equation}

Interestingly, a comparable expression for the Laplace transform of the range seems difficult to compute from \eqref{eq:range_density} and appears to be absent from the literature. The scaling property of Brownian motion evidently implies the following equivalences in law
\begin{equation}\label{eq:range_equivalence}
R(t)\stackrel{d}{=}\sqrt{t}R(1),~~\Theta(y)\stackrel{d}{=}y^2\Theta(1),~~\Theta(y)\stackrel{d}{=}\frac{y^3}{R^2(y)}.
\end{equation}
See Proposition \ref{prop:equivalence} for other similar equivalences in law.

\medskip
We now proceed to the proof of Lemma \ref{lem:min_range}.

\begin{proof}[Proof of Lemma \ref{lem:min_range}]
We start by finding a formula for the characteristic function of the random variable $\Theta(1)-\widetilde{\Theta}(1)$ that is valid for any nonnegative real argument.
Here, $\widetilde{\Theta}(1)$ is an independent copy of $\Theta(1)$. In the rest of the proof we simply write $\Theta$ for $\Theta(1)$, and $\widetilde{\Theta}$ for $\widetilde{\Theta}(1)$.
 From \eqref{eq:rt_density}, we immediately see that the exponential rate of decay of the right-tail of the density of $\Theta$ is $\frac{\pi^2}{2}$. Hence, the Laplace transform $z\mapsto\mathbb{E}[e^{-z \Theta}]$ exists and is analytic in the right half-plane $\mathbb{H}=\{z\in\mathbb{C}:\Real z>-\frac{\pi^2}{2}\}$; see e.g.\ \cite[Chapter II]{Widder}.
Moreover, $\sech^2(\sqrt{z/2})$ is a meromorphic function with no poles in $\mathbb{H}$. This follows from the fact that $\cosh z$ is an entire function whose Taylor expansion contains only even powers and whose zeros occur at the points $i\pi(2k+1)/2$, $k\in\mathbb{Z}$. Therefore, formula \eqref{eq:Lap_transform} for the Laplace transform with $y=1$ is actually valid for all $\lambda \in \mathbb{H}$. 

For $s\geq 0$, we use the identity $\cosh z=\cos iz$ which yields
\begin{align}
\mathbb{E}\left[e^{i 4s (\Theta-\widetilde{\Theta})}\right]&=\mathbb{E}\left[e^{i 4s \Theta}\right]\mathbb{E}\left[e^{-i 4s \Theta}\right]\nonumber\\
&=\sech^2\left((1-i)\sqrt{s}\right)\sech^2\left((1+i)\sqrt{s}\right)\nonumber\\
&=\frac{1}{\cos^2\left((i+1)\sqrt{s}\right)}\frac{1}{\cos^2\left((i-1)\sqrt{s}\right)}.\label{eq:characteristic1}
\end{align}
The identity $\cos(x+iy)=\cos x\cosh y-i\sin x\sinh y$ for $x,y\in\mathbb{R}$ can be used to express the right-hand side of \eqref{eq:characteristic1} as
\begin{align}
&\frac{1}{(\cos \sqrt{s}\cosh \sqrt{s}-i\sin \sqrt{s}\sinh \sqrt{s})^2}\frac{1}{(\cos \sqrt{s}\cosh \sqrt{s}+i\sin \sqrt{s}\sinh \sqrt{s})^2}\nonumber\\
&=\frac{1}{(\cos^2 \sqrt{s}\cosh^2 \sqrt{s}+\sin^2 \sqrt{s}\sinh^2 \sqrt{s})^2}\nonumber\\
&=\frac{1}{(\cosh^2 \sqrt{s}-\sin^2 \sqrt{s})^2}.\label{eq:characteristic2}
\end{align}
The last equality follows from expanding $\cos^2$ and then collapsing the resulting factor $\sinh^2-\cosh^2$  using well-known identities. Further, combining \eqref{eq:characteristic1} and \eqref{eq:characteristic2} along with the double-angle identities for $\cos$ and $\cosh$ leads to the formula
\begin{equation}\label{eq:characteristic}
\mathbb{E}\left[e^{i s (\Theta-\widetilde{\Theta})}\right]=\frac{4}{(\cos\sqrt{s}+\cosh\sqrt{s})^2},\quad s\geq 0.
\end{equation}
The last step of the proof is to use the elementary identity
\[
\min\{x,y\}=\frac{1}{2}(x+y-|x-y|),\quad x,y\in\mathbb{R}
\]
and an integral formula for the absolute moments of a random variable in terms of its characteristic function; see \cite[Lemma 1]{absolute_moments}. In view of \eqref{eq:characteristic} we obtain
\begin{align*}
\mathbb{E}\left[\min\left\{\Theta,\widetilde{\Theta}\right\}\right]&=\mathbb{E}[\Theta]-\frac{1}{2}\mathbb{E}\left[\left|\Theta-\widetilde{\Theta}\right|\right]\\
&=\frac{1}{2}-\frac{1}{\pi}\int_0^\infty\left(1-\frac{4}{(\cos\sqrt{s}+\cosh\sqrt{s})^2}\right)\frac{1}{s^2}\D{s}\\
&=\frac{1}{2}-\frac{2}{\pi}\int_0^\infty \left(1-\frac{4}{(\cos t+\cosh t)^2}\right)\frac{1}{t^3}\D{t},
\end{align*}
where the last equality follows from the change of variables $\sqrt{s}=t$.
\end{proof}

\begin{remark}
We note that an alternative expression for the expected value from Lemma \ref{lem:min_range} can be derived from the density \eqref{eq:rt_density}. However, the resulting infinite series turns out to be rather unwieldy, so we omit it.
\end{remark}

\section{Bounds on the expectation of circumradius and inradius}\label{sec:inradius}
In this section we give detailed proofs for our bounds on the expected values of the circumradius $\calR$ and inradius $\calr$ that are included in Proposition \ref{prop:circumradius_bounds}, Proposition \ref{prop:inradius_bounds} and Proposition \ref{prop:alternative-lower-Inradius}.

\begin{proof}[Proof of Proposition \ref{prop:circumradius_bounds}]
	To prove the lower bound, we employ an elementary inequality $\calD \le 2\calR$, together with \eqref{Jovalekic}. More precisely, we have
	\begin{equation*}
		\mathbb{E}[\calR] \ge \frac{1}{2} \mathbb{E}[\calD] \ge \frac{1.856}{2} = 0.928.
	\end{equation*}
	For the upper bound, we start with the observation that
\begin{equation}\label{eq:circum-proof-bound}
(2\calR)^2 \leq R_1^2 + R_2^2,
\end{equation}
where $R_1$ and $R_2$ are two independent ranges defined in \eqref{Range-process}. Taking expectations, and using Jensen's inequality we obtain
\begin{equation*}
	\mathbb{E}[\calR] \leq \frac{1}{2} \sqrt{\mathbb{E}[R_1^2] + \mathbb{E}[R_2^2]} = \sqrt{2 \log 2},
\end{equation*}
where the last equality follows from \cite{Feller}.

\end{proof}

\begin{proof}[Proof of Proposition \ref{prop:alternative-upper-circumnradius}]
Note that
\begin{equation}\label{eq:Bessel_ball}
	\calR \leq \sup_{0\leq s\leq 1}\|\fatW(s)\| \stackrel{d}{=}\frac{1}{\sqrt{\tau_1}},
\end{equation}
where we used $\tau_1$ to denote the exit time of the unit ball in $\mathbb{R}^2$ by $\fatW$. The equality in distribution appearing in \eqref{eq:Bessel_ball} is similar to those from Proposition \ref{prop:equivalence} and follows from Brownian scaling; see also \cite[Equation (11)]{max_Bessel}. The following identity allows us to compute negative fractional moments of a positive random variable from its Laplace transform. More specifically, for $p>0$, Tonelli's theorem allows us to write 
\begin{equation}\label{eq:moment_formula}
	\int_0^\infty t^{p-1}\mathbb{E}\left[e^{-t\,\tau_1}\right]\D{t}=\mathbb{E}\left[\int_0^\infty t^{p-1}e^{-t\,\tau_1}\D{t}\right]=\Gamma(p)\,\mathbb{E}\left[\tau_1^{-p}\right].
	\end{equation}	
We apply \eqref{eq:moment_formula} to the right-hand side of \eqref{eq:Bessel_ball} by setting $p=1/2$ and using the Laplace transform of $\tau_1$ that was computed in \cite[Equation (32)]{max_Bessel}. In particular, we have
\begin{equation*}
	\mathbb{E}\left[\tau_1^{-1/2}\right]=\Gamma(1/2)^{-1}\int_0^\infty \frac{t^{-1/2}}{I_0(\sqrt{2t)}}\D{t}\approx 1.662180,
\end{equation*}
where $I_\nu(x)$ denotes the modified Bessel function of the first kind of order $\nu$.
\end{proof}


\begin{proof}[Proof of Proposition \ref{prop:inradius_bounds}]
To prove the lower bound, we use a Bonnesen-type inequality from \cite[Equation 15]{Bonnesen} which involves inradius, perimeter, and area, and takes the form
\begin{equation}\label{eq:inradius-proof-bound}
\calr\geq \frac{1}{2\pi}\left(\calP-\sqrt{\calP^2-4\pi \calA}\right).
\end{equation}
Taking expectations while using Jensen's inequality on the right-hand side along with \eqref{eq:perimeter_moment}, \eqref{eq:perimeter_2nd}, and \eqref{eq:area_moment} leads to
\begin{align*}
\mathbb{E}[\calr]&\geq \frac{1}{2\pi}\left(\mathbb{E}[\calP]-\sqrt{\mathbb{E}\left[\calP^2\right]-4\pi \mathbb{E}[\calA]}\right)\\
&=\frac{1}{2\pi}\left(\sqrt{8\pi}-\sqrt{\mathbb{E}\left[\calP^2\right]-2\pi^2}\right)
\approx 0.39305.
\end{align*}

For the upper bound, we start from the obvious inequality $\calA\geq \pi \calr^2$ which yields $\calr\leq\sqrt{\calA /\pi}$. By using Jensen's inequality and \eqref{eq:area_moment} we arrive at
\begin{align*}
\mathbb{E}\left[\calr\right]
\leq\frac{1}{\sqrt{\pi}}\mathbb{E}\left[\sqrt{\calA}\right]
\leq\frac{1}{\sqrt{\pi}}\sqrt{\mathbb{E}[\calA]}
=\frac{\sqrt{2}}{2}
\end{align*}
and the proof is complete.
\end{proof}

\begin{remark}\label{rem:inrad-sim_jensen}
	Using our simulations, we can estimate the expected value of the right-hand side of \eqref{eq:inradius-proof-bound}, that is
	\begin{equation*}
		\frac{1}{2\pi} \mathbb{E} \left[ \left(\calP - \sqrt{\calP^2-4\pi \calA}\right) \right],
	\end{equation*}
	without relying on Jensen's inequality. Monte Carlo estimation of this expectation is $0.419$ which suggests that we do not lose too much with Jensen's inequality. We also estimated the value of $\mathbb{E}[\sqrt{\calA / \pi}]$, and got $0.697$, which is only a bit smaller than $\sqrt{2}/2$.
\end{remark}

\begin{remark}
	It is worth mentioning that using a similar Bonnesen-type inequality we can obtain an alternative upper bound for the expected circumradius. More precisely, taking expectations in \cite[Equation 20]{Bonnesen}, and using Jensen's inequality along with \eqref{eq:perimeter_moment}, \eqref{eq:perimeter_2nd}, and \eqref{eq:area_moment} leads to
\begin{align*}
	\mathbb{E}[\calR]
	& \leq \frac{1}{2\pi} \left(\mathbb{E}[\calP] + \sqrt{\mathbb{E}\left[\calP^2\right]-4\pi \mathbb{E}[\calA]}\right)\\
	& = \frac{1}{2\pi}\left(\sqrt{8\pi} + \sqrt{\mathbb{E}\left[\calP^2\right]-2\pi^2}\right)
\approx 1.2027.
\end{align*}
\end{remark}

\begin{proof}[Proof of Proposition \ref{prop:alternative-lower-Inradius}]
We estimate the inradius of the triangle with vertices $\fatW(0)$, $\fatW(1/2)$, and $\fatW(1)$ that is contained in the convex hull $\calH$; see Figure \ref{fig:Inradius_with_triangle}.
\begin{figure}
\centering
\includegraphics[scale=0.6]{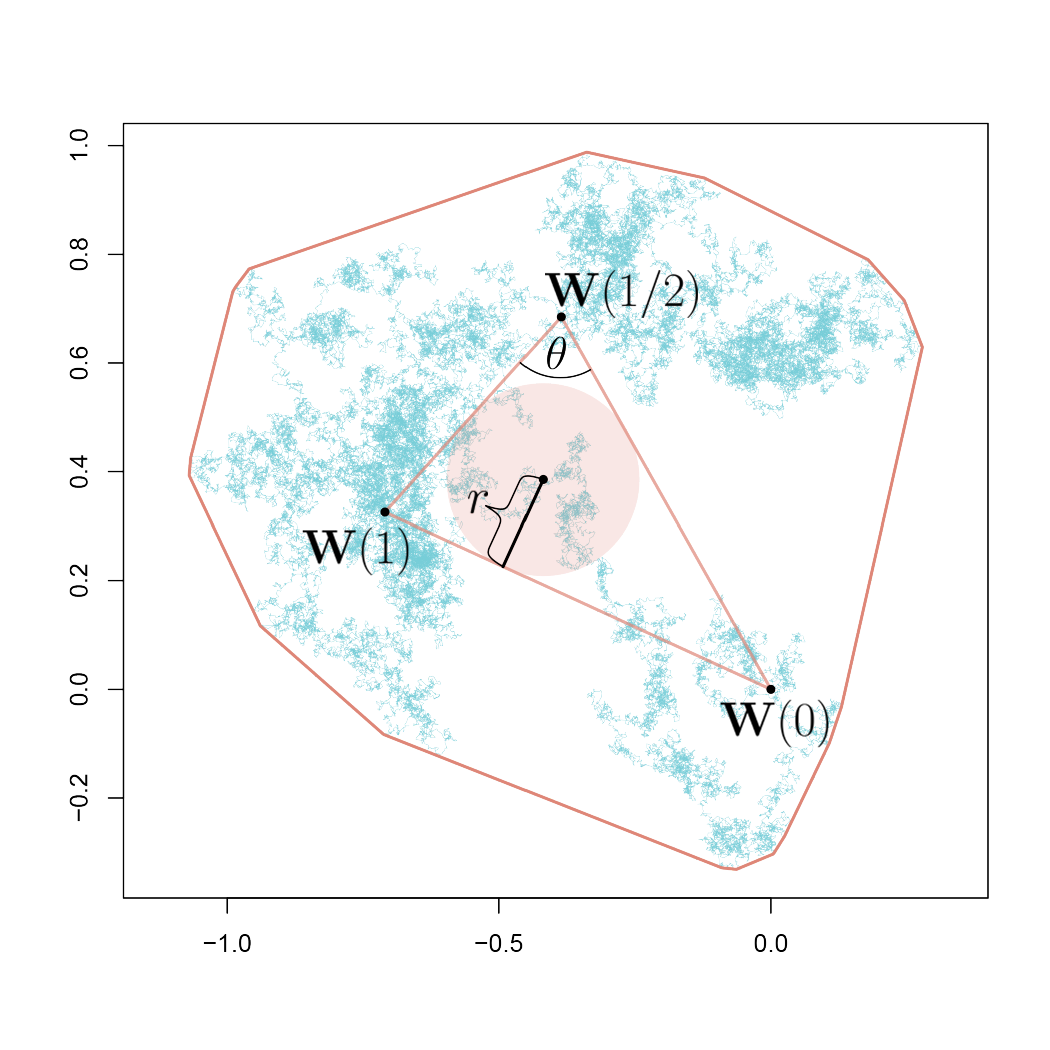}
\vspace{-0.3cm}
\caption{The triangle with vertices $\fatW(0)$, $\fatW(1/2)$, and $\fatW(1)$ contained in $\calH$. The inradius of the triangle is denoted with $r$.}
\label{fig:Inradius_with_triangle}
\end{figure}
More generally, we could choose the intermediate point at some $t\in (0,1)$ instead of at $1/2$, however, $1/2$ apparently provides the best estimate. Let us denote the lengths of the line segments $\overline{\fatW(0) \fatW(1/2)}$, $\overline{\fatW(1/2) \fatW(1)}$, and $\overline{\fatW(0) \fatW(1)}$ by $A$, $B$, and $C$, respectively. Further, we denote the angle $\angle (\fatW(0), \fatW(1/2), \fatW(1))$ by $\theta\in [0,\pi]$. Using the fact that the inradius of a triangle is the quotient of its area and semiperimeter, we obtain 
\begin{equation}\label{eq:triangle_inradius}
\mathbb{E}[\calr]\geq \mathbb{E}\left[\frac{AB\sin \theta}{A+B+C}\right].
\end{equation}
We can use H\"{o}lder's inequality to write
\begin{align*}
\mathbb{E}\left[\sqrt{AB\sin \theta}\right]^2&=\mathbb{E}\left[\sqrt{\frac{AB\sin \theta}{A+B+C}}\sqrt{A+B+C}\right]^2\\
&\leq\mathbb{E}\left[\frac{AB\sin \theta}{A+B+C}\right]\mathbb{E}\left[A+B+C\right],
\end{align*}
and combining this inequality with \eqref{eq:triangle_inradius} results in
\begin{equation}\label{eq:reverse_Holder}
\mathbb{E}[\calr]\geq \frac{\mathbb{E}\left[\sqrt{AB\sin \theta}\right]^2}{\mathbb{E}[A+B+C]}.
\end{equation}

\noindent The rotational invariance and Markov property of planar Brownian motion imply that $A$, $B$, and $\theta$ are mutually independent, that $A$ and $B$ have the same law, and that $\theta$ is uniform on $(0,\pi)$. Moreover, Brownian scaling implies that $C$ has the same law as $\sqrt{2}A$. Since the radial displacement of planar Brownian motion is a $2$-dimensional Bessel process (see  \cite[Chapter XI.1]{Revuz_Yor}), we have 
\[
\mathbb{P}(A\in \D{x})=2xe^{-x^2}\D{x}.
\]
These facts imply
\begin{align}
\mathbb{E}[A+B+C]&=\left(2+\sqrt{2}\right)\mathbb{E}[A]\nonumber 
=2\left(2+\sqrt{2}\right)\int_0^\infty x^2e^{-x^2}\D{x}\nonumber \\
&=\left(2+\sqrt{2}\right)\frac{\sqrt{\pi}}{2}\label{eq:triangle_perimeter}
\end{align}
and
\begin{align}
\mathbb{E}\left[\sqrt{AB\sin \theta}\right]&=\mathbb{E}\left[\sqrt{A}\right]^2\mathbb{E}\left[\sqrt{\sin\theta}\right]\nonumber \\
&=4\left(\int_0^\infty x^\frac{3}{2}e^{-x^2}\D{x}\right)^2\int_0^\pi\frac{\sqrt{\sin u}}{\pi}\D{u}\label{eq:elliptic_integral} \\
&=\Gamma\left(\frac{5}{4}\right)^2\frac{2\sqrt{2}\,\Gamma\left(\frac{3}{4}\right)^2}{\pi^\frac{3}{2}}=\sqrt{\frac{\pi}{8}}\label{eq:duplication}.
\end{align}
The second integral appearing in \eqref{eq:elliptic_integral} can be evaluated using \cite[Identity 3.621.1]{Gradshteyn_Ryzhik}, while the second equality in \eqref{eq:duplication} follows from the Legendre duplication formula. Substituting \eqref{eq:triangle_perimeter} and \eqref{eq:duplication} into \eqref{eq:reverse_Holder} proves the lower bound.
\end{proof}
 \begin{remark}
  Note that using the law of cosines to eliminate $C$ from \eqref{eq:triangle_inradius} and then numerically computing that expectation via the product density of $(A,B,\theta)$ yields a lower bound of approximately $0.1456$, so we do not lose too much using H\"{o}lder's inequality.
 \end{remark}
 
\section{Bounds on the expectations of the inverse processes}\label{sec:inverse-processes-all}

In this section we provide proofs of the bounds for inverse processes presented in the introduction. We start with an auxiliary result which concerns scaling properties and equalities in law for the processes $\calP(t)$, $\calA(t)$, $\calD(t)$, $\calR(t)$, $\calr(t)$ and their inverse processes.

\begin{proposition}\label{prop:equivalence}\
For all $t\geq 0$ and $y>0$, the following equivalences in law hold true. 
\begin{enumerate}[label=(\roman*)]

\item $\displaystyle \calP(t)\stackrel{d}{=}\sqrt{t}\calP(1),~~\Theta^\calP(y)\stackrel{d}{=}y^2\Theta^\calP(1),~~\Theta^\calP(y)\stackrel{d}{=}\frac{y^3}{\calP^2(y)}$;

\item $\displaystyle \calA(t)\stackrel{d}{=}t\,\calA(1),~~\Theta^\calA(y)\stackrel{d}{=}y\,\Theta^\calA(1),~~\Theta^\calA(y)\stackrel{d}{=}\frac{y^2}{\calA(y)}$;

\item $\displaystyle \calD(t)\stackrel{d}{=}\sqrt{t}\calD(1),~~\Theta^\calD(y)\stackrel{d}{=}y^2\Theta^\calD(1),~~\Theta^\calD(y)\stackrel{d}{=}\frac{y^3}{\calD^2(y)}$;

\item $\displaystyle \calR(t)\stackrel{d}{=}\sqrt{t}\calR(1),~~\Theta^\calR(y)\stackrel{d}{=}y^2\Theta^\calR(1),~~\Theta^\calR(y)\stackrel{d}{=}\frac{y^3}{\calR^2(y)}$;

\item $\displaystyle \calr(t)\stackrel{d}{=}\sqrt{t}\calr(1),~~\Theta^\calr(y)\stackrel{d}{=}y^2\Theta^\calr(1),~~\Theta^\calr(y)\stackrel{d}{=}\frac{y^3}{\calr^2(y)}$.
\end{enumerate}
\end{proposition}

\begin{proof}
 We start with a proof for the process $\calA(t)$. The scaling property of Brownian motion implies that 
\begin{equation}\label{eq:scaling_of_Ht}
	\calH(t) =\mathrm{conv} \{ \fatW (s): 0\leq s\leq t\} \stackrel{d}{=} \sqrt{t}\, \mathrm{conv} \{\fatW(s): 0\leq s \leq 1\}=\sqrt{t}\, \calH(1).
\end{equation}
Hence $\calA(t) \stackrel{d}{=} \mathrm{Area}\left(\sqrt{t}\, \mathrm{hull} \{\fatW(s): 0\leq s \leq 1\}\right) = t\, \calA(1)$. Similarly, for any $\lambda > 0$ we have
\begin{align}\label{eq:scaling_for_ar}
	\calA(\lambda t)
	& = \mathrm{Area}(\mathrm{conv} \{\fatW(s) : 0 \leq s \leq \lambda t\}) = \mathrm{Area}(\mathrm{conv} \{\fatW(\lambda s) : 0 \leq s \leq t\}) \nonumber \\
	& \stackrel{d}{=} \mathrm{Area}(\mathrm{conv} \{ \sqrt{\lambda} \fatW(s) : 0 \leq s \leq t\}) = \lambda \calA(t).
\end{align}
This implies
\begin{align*}
\mathbb{P}(\Theta^\calA(y)>t) &= \mathbb{P}(\calA(t)<y) = \mathbb{P}\left(\calA\left(y \cdot \frac{t}{y}\right)<y\right) = \mathbb{P}(y\, \calA(t/y)<y) \\
& = \mathbb{P}(\calA(t/y)<1) = \mathbb{P}(\Theta^\calA(1) > t/y) = \mathbb{P}(y\, \Theta^\calA(1)>t).
\end{align*}
Finally, we have
\begin{align*}
\mathbb{P}\left(\Theta^\calA(y)>t\right) = \mathbb{P}\left(\calA(t)<y\right) = \mathbb{P}\left(\calA\left(y\cdot \frac{t}{y} \right)<y\right) = \mathbb{P}\left(\frac{t}{y}\calA(y)<y\right)=\mathbb{P}\left(\frac{y^2}{\calA(y)}>t\right).
\end{align*}

To prove the first equality in law in (i), (iii), (iv) and (v) we can again use \eqref{eq:scaling_of_Ht} together with the fact that the functionals in question are homogeneous of order one. For completeness, we provide some more details.

For the proof of the first equality in law in (i) we use \eqref{eq:scaling_of_Ht}, which evidently yields $\calP(\lambda t) \stackrel{d}{=} \sqrt{\lambda} \calP(t)$, for any $\lambda > 0$. This further implies
\begin{align*}
	\mathbb{P}(\Theta^{\calP}(y) > t) &= \mathbb{P}(\calP(t) < y) = \mathbb{P} \left(\frac{1}{y} \calP(t) < 1 \right) = \mathbb{P}(\calP(t/y^2) < 1)\\
	& = \mathbb{P}\left( \Theta^{\calP}(1) > \frac{t}{y^2} \right) = \mathbb{P}(y^2 \Theta^{\calP}(1) > t).
\end{align*}
Similarly, we have
\begin{align*}
	\mathbb{P}(\Theta^{\calP}(y) > t) &= \mathbb{P} \left(\frac{1}{y} \calP(t) < 1 \right) = \mathbb{P} \left(\frac{1}{y} \calP\left(y\cdot \frac{t}{y} \right) < 1 \right)\\
	& = \mathbb{P} \left(\sqrt{\frac{t}{y^3}} \calP(y) < 1 \right) 
	= \mathbb{P} \left(\frac{y^3}{\calP^2(y)} > t \right).
\end{align*}

To show the first equality in law in (iii) we simply use $\calD(t) = \sup_{x,y\in \calH(t)}\norm{x-y}$ together with the  Brownian scaling, where $\norm{x}$ stands for Euclidean norm of  vector $x$. The remaining two equivalences can be proved analogously as for the process $\calP(t)$.
 
We now prove (iv). Let $B(x,\rho)$ denote the closed ball in $\mathbb{R}^2$ centered at $x$ and of radius $\rho$. We have
 \begin{align*}
 \calR(t) 
 &= 
 \inf \{ \rho>0: (\exists \, x\in \mathbb{R}^2)\, \calH(t) \subset B(x,\rho)\}\\
 &\stackrel{d}{=}
 \inf \{ \rho>0: (\exists \, x\in \mathbb{R}^2)\, \sqrt{t} \, \calH(1) \subset B(x,\rho)\}\\
 &=
 \inf \{ \rho>0: (\exists \, x\in \mathbb{R}^2)\, \calH(1) \subset B(x,t^{-1/2}\rho)\} = \sqrt{t}\,\calR(1).
 \end{align*}
 We can similarly show that $\calR(\lambda t)  \stackrel{d}{=} \sqrt{\lambda }\, \calR(t)$, for any $\lambda>0$, which implies the other two equivalences from (iv) similarly to the case of (i).
 
Finally, we have
 \begin{align*}
 \calr(t) 
 &= 
 \sup \{ \rho>0: (\exists \, x\in \calH(t))\, B(x,\rho)\subset \calH(t) \}\\
 &\stackrel{d}{=}
 \sup \{ \rho>0: (\exists \, x\in \sqrt{t}\, \calH(1))\, B(x,\rho)\subset \sqrt{t}\, \calH(1) \} \\
 &=
 \sup \{ \rho>0: (\exists \, x\in \calH(1))\, B(\sqrt{t}\, x,\rho)\subset \sqrt{t}\, \calH(1) \} \\
 &=
 \sup \{ \rho>0: (\exists \, x\in \calH(1))\, B( x,t^{-1/2}\rho)\subset  \calH(1) \} = \sqrt{t}\, \calr(1).
 \end{align*}
We can similarly show that $\calr(\lambda t)  \stackrel{d}{=} \sqrt{\lambda }\, \calr(t)$, for any $\lambda>0$. As in the (iv) case, this implies the other two equivalences from (v) and the proof is complete.

%
\end{proof}

\subsection{Bounds on inverse perimeter, area, diameter and circumradius}\label{sec:Inverse-bounds}
We are now ready to prove the series of propositions for inverse processes.

\begin{proof}[Proof of Proposition \ref{prop:inverse-per}]
For the lower bound we simply have
\begin{align*}
\bbE [\Theta^\calP]=\bbE \left[\frac{1}{\calP^2}\right]\geq \frac{1}{\left( \bbE [\calP]\right)^2} = \frac{1}{8\pi},
\end{align*}
where we used Proposition \ref{prop:equivalence} (i), Jensen's inequality and \eqref{eq:perimeter_moment}. 

The upper bound also requires Jensen's inequality, but this time it is used in conjunction with Cauchy's surface area formula. Recalling that formula from \eqref{Cauchy-Per}, we have
\[
\calP = \frac{1}{2}\int_{-\pi}^\pi\underbrace{\sup_{0\leq s \leq 1}\fatW(s)\cdot \mathbf{e}_\theta - \inf_{0\leq s\leq 1}\fatW(s) \cdot \mathbf{e}_\theta}_{\displaystyle w(\theta)}\D{\theta},
\]
where the $\pi$-periodic function $w(\theta)$ is known as the \emph{parametrized width} (or \emph{parametrized range}) of the convex hull $\calH$. It follows that Cauchy's surface area formula can be equivalently stated as 
\begin{equation}\label{eq:Cauchy}
\calP=\int_0^\pi w(\theta)\D\theta.
\end{equation}
By interpreting the following integrals as expected values with respect to the uniform distribution on $(0,\pi)$, we can use Jensen's inequality to deduce that 
\begin{equation*}
\frac{1}{\pi}\int_0^\pi\frac{1}{w^2(\theta)}\D\theta\geq\frac{1}{\left(\frac{1}{\pi}\int_0^\pi w(\theta)\D\theta\right)^2}.
\end{equation*}
In particular, now Cauchy's surface area formula \eqref{eq:Cauchy}  implies that
\begin{equation}\label{eq:Cauchy_Jensen}
\frac{1}{\calP^2}\leq\frac{1}{\pi^3}\int_0^\pi\frac{1}{w^2(\theta)}\D\theta.
\end{equation}
By rotational invariance of Brownian motion, the law of $w(\theta)$ does not depend on $\theta$. Hence we can combine \eqref{eq:Cauchy_Jensen} with Proposition \ref{prop:equivalence} (i) and apply Tonelli's theorem to deduce that
\begin{align*}
\mathbb{E}\left[\Theta^\calP\right]
=\mathbb{E}\left[\frac{1}{\calP^2}\right]\leq\frac{1}{\pi^2}\mathbb{E}\left[\frac{1}{w^2(0)}\right].
\end{align*}
Recall from \eqref{Range-process} that $w(0)$ has the same law as the range $R(1)$ of the one-dimensional Brownian motion run up to time $1$. Now it follows from \eqref{eq:range_equivalence} and \eqref{eq:Lap_transform} that 
\begin{align*}
\mathbb{E}\left[\Theta^\calP\right]
\leq\frac{1}{\pi^2}\mathbb{E}\left[\frac{1}{R^2(1)}\right]
=\frac{1}{\pi^2}\mathbb{E}\left[\Theta (1)\right]
=\frac{1}{2\pi^2},
\end{align*}
and the proof is finished.
\end{proof}

\begin{proof}[Proof of Proposition \ref{prop:inverse-area}]

For the lower bound we write
\begin{align*}
\bbE [\Theta^\calA]=\bbE \left[ \frac{1}{\calA}\right] \geq \frac{1}{\bbE [\calA]}=\frac{2}{\pi},
\end{align*}
where we applied Proposition \ref{prop:equivalence} (ii), Jensen's inequality and \eqref{eq:area_moment}. 

For the upper bound, we employ a two-stage construction of a stopping time at which the convex hull is guaranteed to contain two adjacent triangles of a certain total area. First we wait until one of the coordinates of $\fatW$ has range $a$, where $a > 0$ is a parameter that will be optimized later. This waiting time is distributed like the random variable $\min\{\Theta^{R_1}(a),\, \Theta^{R_2}(a)\}$ which was investigated in Section \ref{sec:Range}. Moreover, at this time the convex hull will contain a line segment $\overline{x}$ of length at least $a$ with $\fatW$ being located at one of the endpoints of $\overline{x}$. 

Next, we wait until the orthogonal projection of $\fatW$ on the affine subspace spanned by $\overline{x}^{\perp}$ attains range $b$, where $b > 0$ is another parameter to be optimized later. This waiting time has the same distribution as $\Theta(b)$. It follows that after waiting for both of these events to occur in turn, the convex hull of the planar Brownian motion will contain two adjacent triangles with base at least $a$ and heights $c$ and $b-c$, for some $c \in (0, b)$; see Figure \ref{fig:area-strip}. 

\begin{figure}
	\centering
	\begin{subfigure}{.5\textwidth}
		\centering
		\includegraphics[scale = 0.45]{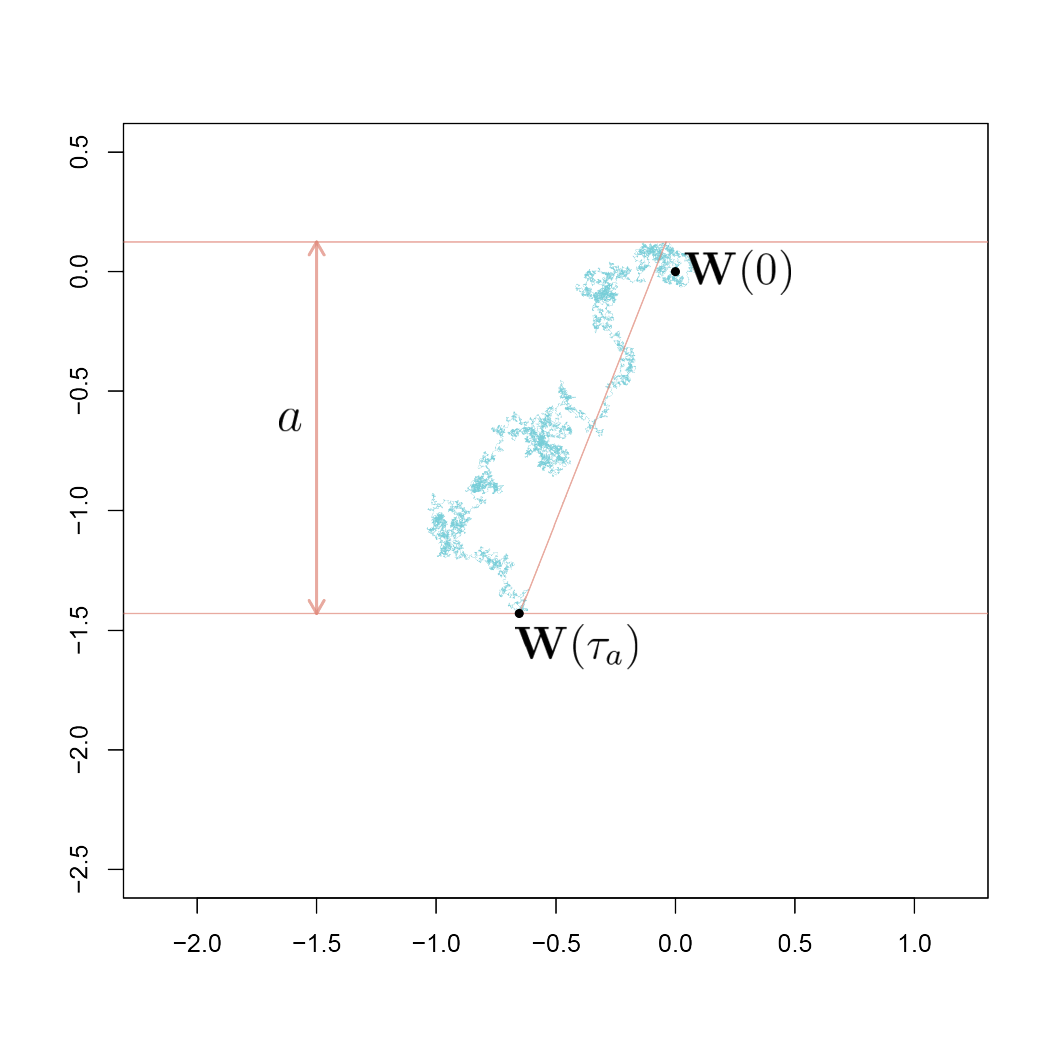}
		\caption{First step.}\label{fig:Area_UB_step_1}
	\end{subfigure}%
	\hfill
	\begin{subfigure}{.5\textwidth}
		\centering
		\includegraphics[scale = 0.45]{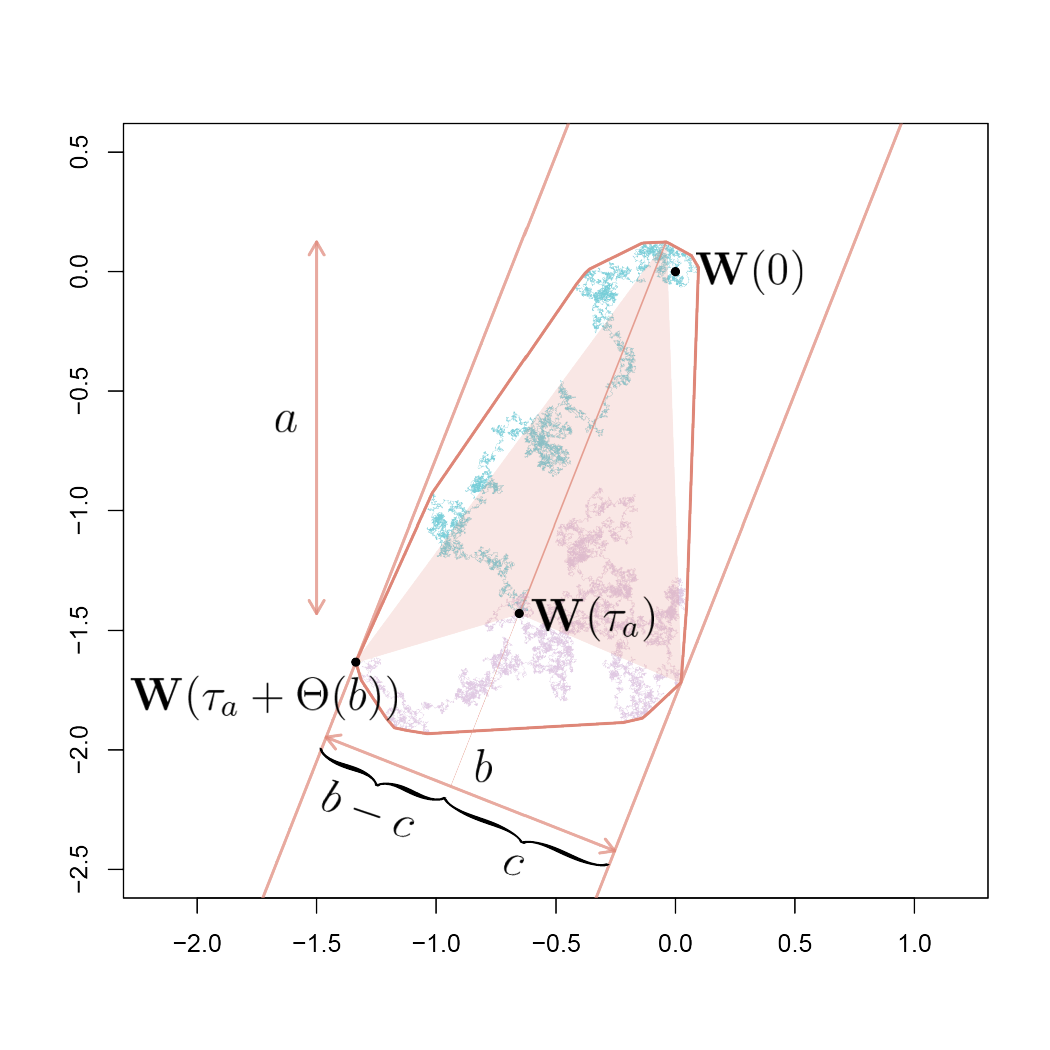}
		\caption{Second step.}\label{fig:Area_UB_step_2}
	\end{subfigure}%
	\caption{In the first step we wait until one of the coordinate range processes hits value $a$. For simplicity, this stopping time is denoted by $\tau_a$, that is, $\tau_a = \min\{\Theta^{R_1}(a), \Theta^{R_2}(a)\}$. The trajectory up to this random time is colored blue. In the second step, we wait until the orthogonal projection of $\fatW$ on the affine subspace spanned by $\overline{x}^{\perp}$ attains range $b$. The trajectory from the time point at which one of the components attained range $a$ till the time when this orthogonal projection reaches range $b$ is colored purple. In the end, we are sure that two adjacent triangles with combined area of $ab/2$ will be contained in the convex hull spanned by the trajectory.} \label{fig:area-strip}
\end{figure}

Hence, the area of the convex hull at this time will be at least $ab/2$. This implies that 
\[
\mathbb{E}\left[\Theta^\calA(ab/2)\right]\leq \mathbb{E}\left[\min\left\{\Theta^{R_1}(a),\, \Theta^{R_2}(a)\right\}\right]+\mathbb{E}\left[ \Theta(b) \right].
\]
Now it follows from \eqref{eq:range_equivalence} and \eqref{eq:Lap_transform} that
\begin{equation}\label{eq:area_inequality}
\mathbb{E}\left[\Theta^\calA(ab/2)\right]\leq a^2\, \mathbb{E}\left[\min\left\{\Theta^{R_1},\, \Theta^{R_2}\right\}\right]+\frac{b^2}{2}.
\end{equation}
Minimizing the right-hand side of \eqref{eq:area_inequality} under the constraints $ab/2=1$, $a>0$, and $b>0$ results in the desired upper bound
\[
\mathbb{E}\left[\Theta^\calA\right]\leq 2 \sqrt{2\mathbb{E}\left[\min\left\{\Theta^{R_1},\, \Theta^{R_2}\right\}\right]}=2\sqrt{2\mathfrak{r}},
\] 
and the proof is finished.
\end{proof}

\begin{proof}[Proof of Proposition \ref{prop:inverse-diam}]
Proposition \ref{prop:equivalence} (iii) and Jensen's inequality together imply that 
\begin{equation}\label{eq:Jensen_diameter}
\mathbb{E}\left[\Theta^\calD\right]\geq\frac{1}{\left(\mathbb{E}\left[\calD\right]\right)^2}.
\end{equation}
Combining \eqref{eq:Jensen_diameter} with the upper bound from \eqref{bound-McRedmond-Xu}
yields the desired lower bound. 

The upper bound follows from the fact that as soon as one of the processes $W_1(t)$,  $W_2(t)$ (one of the coordinates of the planar Brownian motion $\fatW$) attains range $1$, then the diameter of $\calH(t)$ is at least $1$.
\end{proof}

\begin{proof}[Proof of Proposition \ref{prop:inverse-circumradius}]
	Proposition \ref{prop:equivalence} (iv) and Jensen's inequality together imply that 
\begin{equation}\label{eq:Jensen_circumradius}
	\mathbb{E}\left[\Theta^\calR\right]\geq\frac{1}{\left(\mathbb{E}\left[\calR\right]\right)^2}.
\end{equation}
Combining \eqref{eq:Jensen_circumradius} with the upper bound from \eqref{bounds-circumradius} yields the desired lower bound. 
	
	For the upper bound we observe that, as soon as one of the two independent range processes hits value $2$, the convex hull contains a line segment of length at least $2$. Hence at this time point the circumradius must be at least $1$.
\end{proof}

\subsection{Radial slit plane}\label{sec:slit}
Before we can give the proof of Proposition \ref{prop:inverse-inrad}, we first need to prove Theorem \ref{lem:slit_exit}.
We start  with a related computational result on the second moment of the radial displacement of the Brownian motion when it exits the radial slit plane $\mathcal{S}_n$.

\begin{lemma}\label{lem:radial_moment}
Let $\mathcal{S}_n$ be the radial slit plane defined in \eqref{eq:slit_plane} and let $\tau_{\mathcal{S}_n}$ denote the exit time of $\mathcal{S}_n$ by the planar Brownian motion $\fatW$ started at the origin. Then it holds that
\[
\mathbb{E}\left[\norm{\fatW(\tau_{\mathcal{S}_n})}^2\right]=\begin{cases}
\infty&n\leq 4\\
\displaystyle\frac{\Gamma\left(\frac{1}{2}-\frac{2}{n}\right)}{\sqrt{\pi}\,\Gamma\left(1-\frac{2}{n}\right)}&n\geq 5.
\end{cases}
\]
\end{lemma}

\begin{proof}[Proof of Lemma \ref{lem:radial_moment}]
By \cite[Example 2.1]{harmonic_survey}, we have 
\[
\mathbb{P}\Big(\norm{\fatW(\tau_{\mathcal{S}_n})}\leq r\Big)=\frac{2}{\pi}\arctan\sqrt{r^n-1},\quad r\geq 1.
\] 
Hence,
\begin{equation}\label{eq:second_moment}
\mathbb{E}\left[\norm{\fatW(\tau_{\mathcal{S}_n})}^2\right]=\frac{n}{\pi}\int_1^\infty \frac{r}{\sqrt{r^n-1}}\D{r}.
\end{equation}
Since 
\begin{align*}
\frac{r}{\sqrt{r^n-1}}= 
\begin{cases}
O(1/\sqrt{r-1}),&\mathrm{as}\ r\searrow 1;\\ 
O(r^{1-n/2}),& \mathrm{as}\ r\to\infty,
\end{cases}
\end{align*}
it is clear from \eqref{eq:second_moment} that $\mathbb{E}\big[\norm{\fatW(\tau_{\mathcal{S}_n})}^2\big ]=\infty$ if and only if $n\leq 4$. 

We next compute $\mathbb{E}\big[\norm{\fatW(\tau_{\mathcal{S}_n})}^2\big]$ explicitly for $n\geq 5$. For this we apply the change of variables $r= \sec^{2/n} u$ in \eqref{eq:second_moment}. This results in
\begin{align}
\mathbb{E}\left[\norm{\fatW(\tau_{\mathcal{S}_n})}^2\right]&=\frac{n}{\pi}\int_0^\frac{\pi}{2} \frac{\sec^\frac{2}{n} u}{\sqrt{\sec^2 u-1}}\frac{2\sec^\frac{2}{n}u\tan u}{n}\D{u}\nonumber \\
&=\frac{2}{\pi}\int_0^\frac{\pi}{2}\cos^{-\frac{4}{n}}u\D{u}\label{eq:GR_integral}\\
&=\frac{\Gamma\left(\frac{1}{2}-\frac{2}{n}\right)^2}{2^\frac{4}{n}\pi\,\Gamma\left(1-\frac{4}{n}\right)}\nonumber \\
&=\frac{\Gamma\left(\frac{1}{2}-\frac{2}{n}\right)}{\sqrt{\pi}\,\Gamma\left(1-\frac{2}{n}\right)},\quad n\geq 5.\label{eq:L_duplication}
\end{align}
The integral appearing in \eqref{eq:GR_integral} can be evaluated using \cite[Identity 3.621.1]{Gradshteyn_Ryzhik}, while the equality in \eqref{eq:L_duplication} follows from the Legendre duplication formula.

\end{proof}

Now we are ready to prove Theorem \ref{lem:slit_exit}.

\begin{proof}[Proof of Theorem \ref{lem:slit_exit}]
First we show that $\mathbb{E}\left[\tau_{\mathcal{S}_n}\right]=\infty$ when $n\leq 4$. Recall that $\fatW(t)=(W_1(t),W_2(t))$, where $W_1(t)$ and $W_2(t)$ are two independent one-dimensional Brownian motions started at zero. Moreover, each process $W_i(t)$ is a Brownian motion in the natural filtration of $\fatW$, with respect to which $\tau_{\mathcal{S}_n}$ is an almost surely finite stopping time. In particular, the optional stopping theorem applied to the martingale $W_1^2(t)+W_2^2(t)-2t$ implies that for each $j\in\mathbb{N}$ we have
\[
\mathbb{E}\left[W_1^2(\tau_{\mathcal{S}_n}\wedge j)+W_2^2(\tau_{\mathcal{S}_n}\wedge j)\right]=2\,\mathbb{E}\left[\tau_{\mathcal{S}_n}\wedge j\right]\leq 2\,\mathbb{E}\left[\tau_{\mathcal{S}_n}\right].
\]
Letting $j\to\infty$ in the above inequality while using Fatou's lemma on the left-hand side leads to
\begin{equation*}
\mathbb{E}\left[W_1^2(\tau_{\mathcal{S}_n})+W_2^2(\tau_{\mathcal{S}_n})\right]\leq 2\,\mathbb{E}\left[\tau_{\mathcal{S}_n}\right].
\end{equation*}
Now Lemma \ref{lem:radial_moment} implies that for $n\leq 4$ we have
\[
\mathbb{E}\left[\tau_{\mathcal{S}_n}\right]\geq \frac{1}{2}\mathbb{E}\left[\norm{\fatW(\tau_{\mathcal{S}_n})}^2\right]=\infty.
\]

Next we show that $\mathbb{E}\left[\tau_{\mathcal{S}_n}\right]<\infty$ when $n\geq 5$. Towards this end, note that for each $n\in\mathbb{N}$ we have $\mathcal{S}_n\subset\mathcal{S}_1$, hence
\begin{equation}\label{eq:monotonicity}
\tau_{\mathcal{S}_n}\preceq\tau_{\mathcal{S}_1},~n\in\mathbb{N},
\end{equation}
where $\preceq$ denotes the usual stochastic order. Since $\mathcal{S}_1$ is simply a translation of a degenerate wedge with total opening angle $2\pi$, Spitzer's well-known integrability condition for the exit time of wedges \cite[Theorem 2]{Spitzer} shows that $\mathbb{E}\left[\tau^p_{\mathcal{S}_1}\right]<\infty$ when $p<\frac{1}{4}$. In particular, along with \eqref{eq:monotonicity}, this implies that
\begin{equation}\label{eq:log_moment}
\mathbb{E}\left[\log\tau_{\mathcal{S}_n}\right]\leq\mathbb{E}\left[\log\tau_{\mathcal{S}_1}\right]<\infty,~n\in\mathbb{N}.
\end{equation}
Denote the maximal process of $\fatW$ by $W^*=\{W^*(t):t\geq 0\}$, that is, 
\[
W^*(t):=\sup_{0\leq s\leq t}\norm{\fatW(s)},\quad t\geq 0.
\]
Then for each $n\in\mathbb{N}$, the bound \eqref{eq:log_moment} allows us to apply \cite[Theorem 2.2]{Burkholder} to deduce that 
\begin{equation}\label{eq:maximal_radial}
\mathbb{E}\left[\left(W^*(\tau_{\mathcal{S}_n})\right)^2\right]\leq c\, \mathbb{E}\left[\norm{\fatW(\tau_{\mathcal{S}_n})}^2\right],
\end{equation}
for some constant $c\in(0,\infty)$. When $n\geq 5$, we can use Lemma \ref{lem:radial_moment} along with \eqref{eq:maximal_radial} to conclude that 
\[
\mathbb{E}\left[\left(W^*(\tau_{\mathcal{S}_n})\right)^2\right]<\infty.
\]
Now the finiteness of $\mathbb{E}\left[\tau_{\mathcal{S}_n}\right]$ for $n\geq 5$ follows from \cite[Theorem 2.1]{Burkholder}.

Finally, we compute $\mathbb{E}\left[\tau_{\mathcal{S}_n}\right]$ when $n\geq 5$. Since $\mathbb{E}\left[\tau_{\mathcal{S}_n}\right]<\infty$ in this case, we can use Wald's second lemma \cite[Theorem 2.48]{Morters_Peres} to obtain
\begin{align*}
\mathbb{E}\left[\norm{\fatW(\tau_{\mathcal{S}_n})}^2\right]
=\mathbb{E}\left[W_1^2(\tau_{\mathcal{S}_n})\right]+\mathbb{E}\left[W_2^2(\tau_{\mathcal{S}_n})\right]
=2\,\mathbb{E}\left[\tau_{\mathcal{S}_n}\right],
\end{align*}
and the desired result follows from Lemma \ref{lem:radial_moment}.
\end{proof}

\subsection{Bounds on inverse inradius}\label{sec:inverse-inradius-bounds}
We can finally give the proof of Proposition \ref{prop:inverse-inrad}.

\begin{proof}[Proof of Proposition \ref{prop:inverse-inrad}]
To prove the lower bound, note that the classical isoperimetric inequality implies that at time $\Theta^\calA(\pi)$, the inradius must be less than or equal to $1$. Hence $\Theta^\calA(\pi)\leq\Theta^\calr(1)$. It follows from Proposition \ref{prop:equivalence} (ii) and Proposition \ref{prop:inverse-area} that
\begin{align*}
\mathbb{E}\left[\Theta^\calr\right] 
\geq \mathbb{E}\left[\Theta^\calA(\pi)\right] 
=\pi\,\mathbb{E}\left[\Theta^\calA\right]
\geq 2.
\end{align*}

For the upper bound, we employ a two-stage construction similar to that from Proposition \ref{prop:inverse-area}. More precisely, we first wait until one of the coordinates of $\fatW$ has range $a$, where $a>0$ is a parameter that will be optimized later. This waiting time is distributed like the random variable $\min\{\Theta^{R_1}(a),\, \Theta^{R_2}(a)\}$. Moreover, at this time the convex hull will contain a line segment of length at least $a$ with the Brownian motion being located at one of the endpoints of this line segment. Without loss of generality, we place this endpoint at the origin with the line segment oriented vertically; see Figure \ref{fig:6-slit_domain}.

In the second stage of the construction, we wait until the Brownian motion, which starts at the origin, exits a radial slit plane with $6$ slits and radius $r$, where $r$ is a parameter. Depending on which slit the Brownian motion hits upon exiting, the convex hull will now contain one of the three types of triangles labeled $A$, $B$, or $C$ in Figure \ref{fig:radial_slits}. 
The area, perimeter, and inradius of these triangles are straightforward to calculate and are listed in Table \ref{Table:triangles}.

\begin{figure}[h]
\centering
\includegraphics[scale=.35]{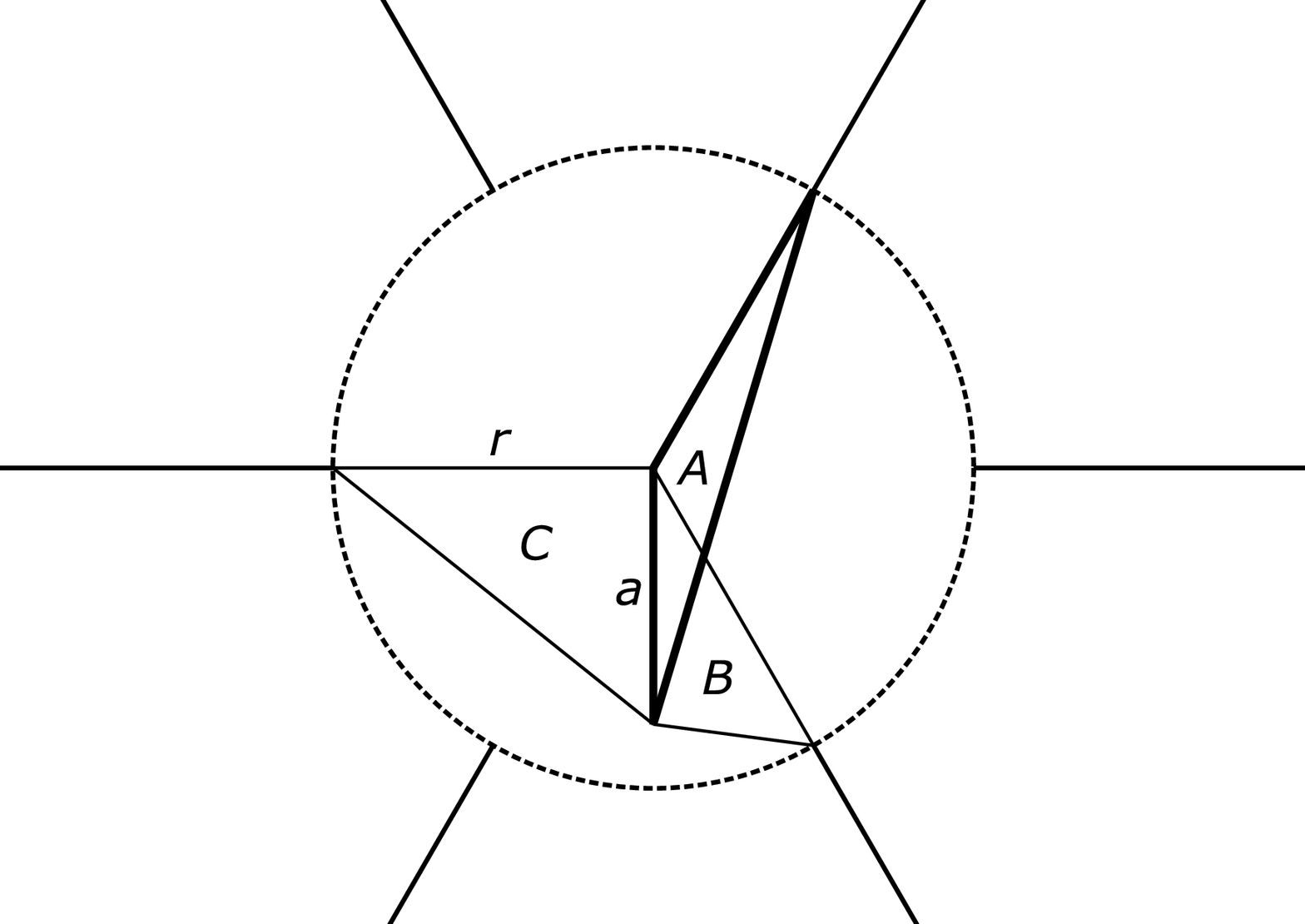}
\caption{A radial slit plane $r\mathcal{S}_n$ with $n=6$ and radius $r$. The top of the initial segment of length $a$ is located at the origin. The resulting types of triangles are labeled $A$, $B$, and $C$, with type $A$ outlined in bold.}
\label{fig:radial_slits}
\end{figure}

\begin{table}[h]
\begin{center}
\renewcommand{\arraystretch}{1.4}
\begin{tabular}{|c|c|c|c|}
\hline
Triangle & Area & Perimeter & Inradius\\
\hline \hline
$A$ & $\frac{1}{4}ar$ & $a+r+\sqrt{a^2+r^2+\sqrt{3}a r}$ & $\frac{ar/2}{a+r+\sqrt{a^2+r^2+\sqrt{3}a r}}$\\
\hline
$B$ & $\frac{1}{4}ar$ & $a+r+\sqrt{a^2+r^2-\sqrt{3}a r}$ & $\frac{ar/2}{a+r+\sqrt{a^2+r^2-\sqrt{3}a r}}$\\
\hline
$C$ & $\frac{1}{2}ar$ & $a+r+\sqrt{a^2+r^2}$ & $\frac{ar}{a+r+\sqrt{a^2+r^2}}$\\
\hline
\end{tabular}
\renewcommand{\arraystretch}{1}
\end{center}
\caption{Three triangles $A$, $B$ and $C$}
\label{Table:triangles}
\end{table}
It is clear from the table that type $A$ triangles always have the smallest inradius regardless of the parameters $a$ and $r$. Hence, after the second stage is completed, the inradius of the convex hull of the planar Brownian motion will be at least that of a type $A$ triangle.
\newpage 
This implies that 
\[
\mathbb{E}\left[\Theta^\calr(\varrho(a,r))\right]
\leq 
\mathbb{E}\left[\min\{\Theta^{R_1}(a),\, \Theta^{R_2}(a)\}\right]+\mathbb{E}\left[\tau_{r\mathcal{S}_6}\right], 
\]
where $\varrho (a,r) = \frac{ar/2}{a+r+\sqrt{a^2+r^2+\sqrt{3}a r}}$.
Further, it follows from the scaling relation \eqref{eq:range_equivalence} and Theorem \ref{lem:slit_exit} that
\begin{equation}\label{eq:inradius_time}
\mathbb{E}\left[\Theta^\calr(\varrho(a,r))\right]
\leq 
a^2\, 
\mathbb{E}\left[\min\{\Theta^{R_1},\Theta^{R_2}\}\right]+\frac{\Gamma\left(\frac{1}{6}\right)}{2\sqrt{\pi}\,\Gamma\left(\frac{2}{3}\right)}r^2.
\end{equation}

It remains to minimize the right-hand side of \eqref{eq:inradius_time} under the constraints $a>0$, $r>0$, and
$ \varrho (a,r) =1$.
Noting that these constraints imply
\begin{equation}\label{eq:inradius_constraint}
ar/2-a-r=\sqrt{a^2+r^2+\sqrt{3}a r}\geq \max \{a,r\},
\end{equation}
we are led to the further constraints $a>4$ and $r>4$, which are valid due to the fact that for triangles of type $A$ and $B$ (see Figure \ref{fig:radial_slits} and Table \ref{Table:triangles}), fixing $a$ and letting $r \to \infty$, or vice-versa, results in an inradius of the fixed parameter divided by $4$. Moreover, by taking the square of the equality in  \eqref{eq:inradius_constraint}, we see that
\begin{equation}\label{eq:inradius_equation}
r=4\frac{a-\left(2-\sqrt{3}\right)}{a-4}.
\end{equation}
Substituting \eqref{eq:inradius_equation} into \eqref{eq:inradius_time} and using Lemma \ref{lem:min_range} yields the upper bound
\[
\mathbb{E}\left[\Theta^\calr\right]\leq 0.3466\, a^2+8\frac{\Gamma\left(\frac{1}{6}\right)}{\sqrt{\pi}\,\Gamma\left(\frac{2}{3}\right)}\left(\frac{a-\left(2-\sqrt{3}\right)}{a-4}\right)^2,~a>4.
\]
We can optimize this bound numerically over the constraint $a>4$ using the command \texttt{NMinimize} in \emph{Mathematica}. This results in the desired upper bound 
\[
\mathbb{E}\left[\Theta^\calr\right]\leq 83.40
\]
which is achieved at $a\approx 9.79$.

\begin{figure}
	\centering
	\begin{subfigure}{.5\textwidth}
		\centering
		\includegraphics[scale = 0.45]{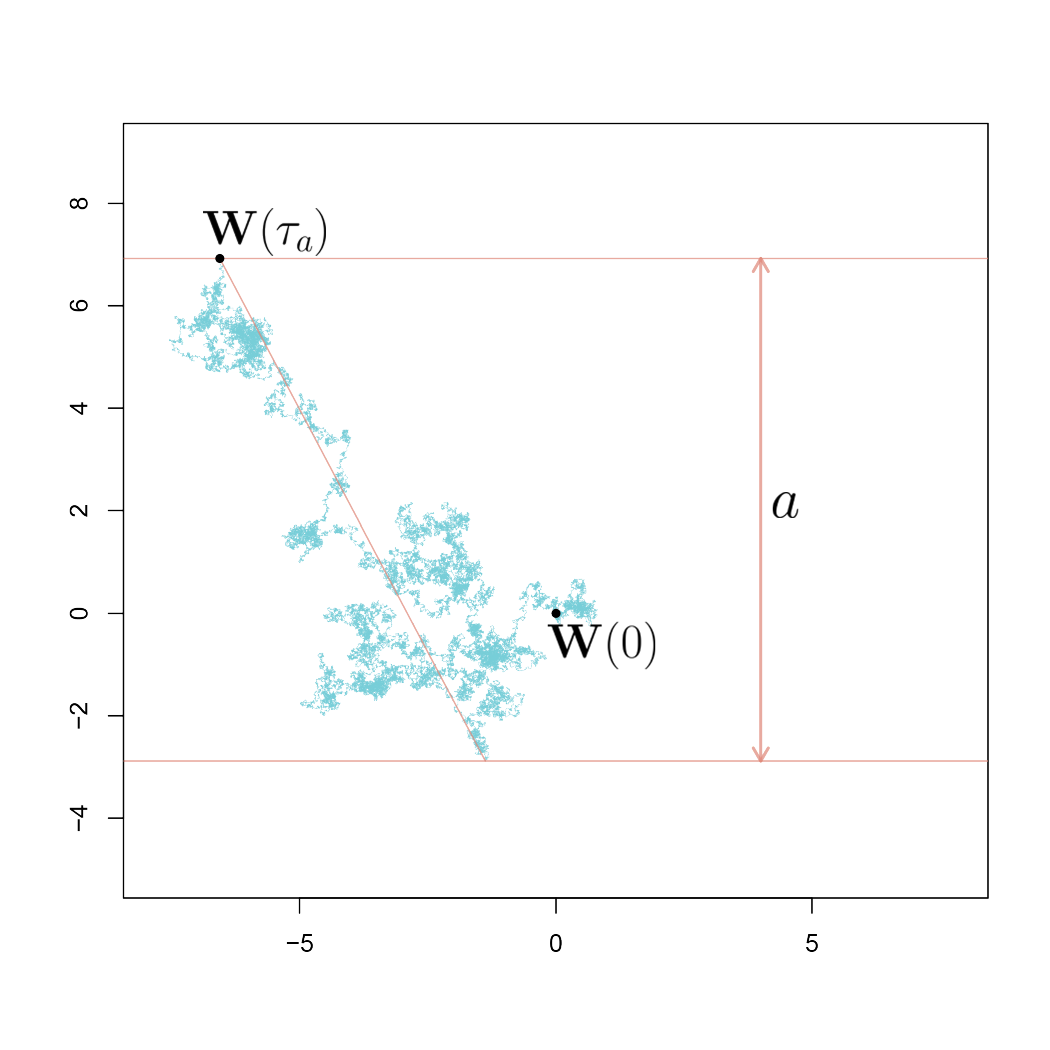}
		\caption{First step.}\label{fig:Slit_domain_step_1}
	\end{subfigure}%
	\hfill
	\begin{subfigure}{.5\textwidth}
		\centering
		\includegraphics[scale = 0.45]{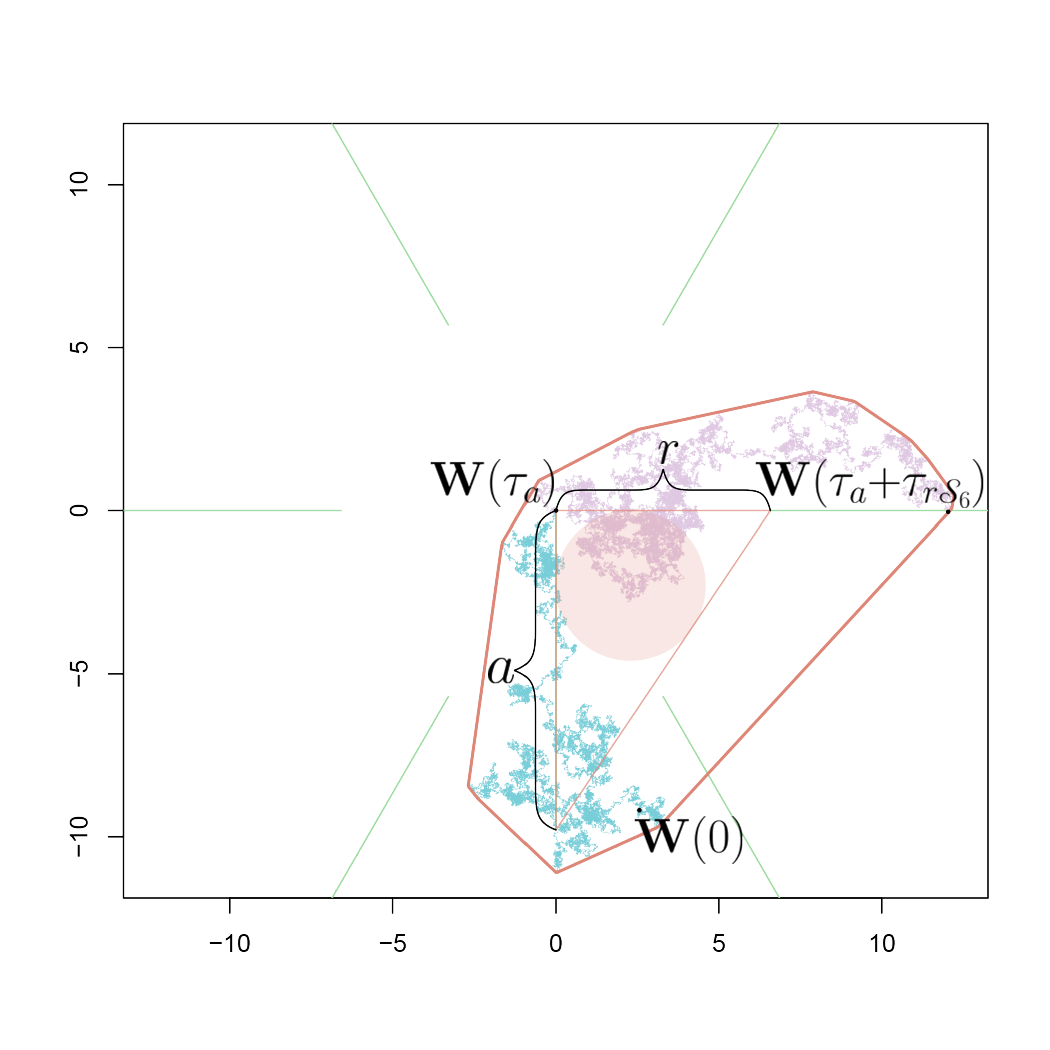}
		\caption{Second step.}\label{fig:Slit_domain_step_2}
	\end{subfigure}%
	\caption{In the first step we wait until one of the coordinate range processes hits value $a$. For simplicity, this stopping time is denoted by $\tau_a$, that is, $\tau_a = \min\{\Theta^{R_1}(a), \Theta^{R_2}(a)\}$. The trajectory up to this random time is colored blue. At this random time, the convex hull contains a line segment of length at least $a$ and $\fatW$ is positioned at one of its endpoints. Now we rotate and translate the trajectory so that this endpoint is at the origin and the mentioned line segment is oriented vertically. In the second step, we wait until $\fatW$ exits a radial slit plane with $6$ slits (colored green) and radius $r$. The trajectory corresponding to the second step is colored purple. In the end, one of the triangles shown in Figure \ref{fig:radial_slits} will be contained in the convex hull of the trajectory.}\label{fig:6-slit_domain}
\end{figure}
\end{proof}

\section{A few words on Monte Carlo simulations}\label{sec:simulations}

In this section we provide a few more details about the simulations we performed to obtain estimates of the expected values of all the random variables that we are studying in this paper. In addition to giving us the estimates of the true value of those expected values, the simulations also give us information on how sharp the bounds we have obtained are, and where one has the most space for improvement, cf.\ Table \ref{Table-simulations}. To be able to extract more information from the simulations, for each simulation of the trajectory of the Brownian motion we calculated values of all the variables we were interested in --- perimeter, area, diameter, circumradius and inradius of the convex hull of that trajectory run up to time $1$, and the time when the perimeter, area, diameter, circumradius and inradius of the convex hull of that trajectory attained value $1$ for the first time. In this way we have simulated realizations of the random vector that comprises of all those random variables. This enabled us to estimate some other interesting quantities such as the one appearing in Remark \ref{rem:inrad-sim_jensen}. We simulated $10^5$ trajectories of Brownian motion, so we had $10^5$ realizations of the mentioned random vector. Moreover, the simulated trajectories each had time increments of size $10^{-7}$.

To perform the simulations we used Python \cite{Python}. After simulating Brownian motion and calculating the convex hull (which can be done with built-in functions), we had to calculate the functionals that we are interested in --- perimeter, area, diameter, circumradius and inradius. Since the function for calculating the convex hull returns extremal points of the convex hull (in clockwise order), it is then straightforward to calculate perimeter, area and diameter. Calculating the circumradius and inradius is, however, a bit more challenging. Luckily, for the circumradius, one can apply the well-known Welzl's algorithm; see \cite{welzl}. We include a few more words on the issue of calculating the inradius in the following subsection.

\subsection{Inradius and Chebyshev's center}\label{sec:inrad}

To obtain a Monte Carlo estimation of the expected value of the inradius of the convex hull of the trajectory of the planar Brownian motion up to time $1$, and the expected value of the time when the inradius of the convex hull of the trajectory of planar Brownian motion attains value $1$, we make use of the notion of Chebyshev's center for sets. Below we provide a brief definition of Chebyshev's center in a general setting, and then we explain how we used it in the present framework. For what follows we refer to \cite[Section 8.5.1]{Boyd}.

Let $C \subset\mathbb{R}^n$ be a bounded set with nonempty interior. A Chebyshev's center is a center of the largest ball that lies inside $C$. Figure \ref{fig:Incircle} shows an example where the set $C$ is exactly the convex hull of the trajectory of planar Brownian motion up to time $1$.


When the set $C$ is convex, computing Chebyshev's center is a convex optimization problem. More specifically, suppose $C \subset \mathbb{R}^n$ is defined by a set of inequalities:
\begin{equation*}
	C = \{x \in \mathbb{R}^n \mid f_1(x) \le 0, \ldots, f_m(x) \le 0\},
\end{equation*}
where the functions $f_1,\ldots ,f_m$ are all convex.
A Chebyshev's center of $C$ is defined as a solution to the following optimization problem
\begin{equation}\label{eq:opt_problem}
	\begin{aligned}
		\textnormal{maximize} & \quad \rho \\
		\textnormal{subject to} & \quad g_i(x, \rho)\le 0, \quad i = 1, \ldots, m,
	\end{aligned}
\end{equation}
where $g_i$ is defined as
\begin{equation*}
	g_i(x, \rho) = \sup_{\norm{u} \le 1} f_i(x + \rho u).
\end{equation*}
Since each function $g_i$ is the pointwise maximum of a family of convex functions, it is itself convex, and thus problem \eqref{eq:opt_problem} is a convex optimization problem,  However, evaluating $g_i$ involves solving a convex maximization problem (either numerically or analytically), which may be very hard. 
In practice, one can find Chebyshev's center only in cases when the functions $g_i$ are easy to evaluate. Fortunately, for the sake of our simulations, we only need to find Chebyshev's center for polygons in $\mathbb{R}^2$.

If the set $C$ is a polyhedron in $\mathbb{R}^d$, i.e., it is defined by a set of linear inequalities $a_i^T x \le b_i$, $i = 1, 2, \ldots, m$ (where $a^T$ is the transpose of vector $a$), we have 
\begin{equation*}
	g_i(x, \rho) = \sup_{\norm{u} \le 1} a_i^T (x + \rho u) - b_i = a_i^T x + \rho \norm{a_i} - b_i.
\end{equation*}
Hence,  Chebyshev's center can be found by solving the following linear problem
\begin{equation*}
	\begin{aligned}
		\textnormal{maximize} & \quad \rho \\
		\textnormal{subject to} & \quad a_i^T x + \rho \norm{a_i} \le b_i, \quad i = 1, \ldots, m \\
		& \quad \rho \ge 0
	\end{aligned}
\end{equation*}
with variables $x$ and $\rho$.

\subsection{Numerical results}

Even though we could extract more information from our simulations, we give only a brief summary of our results; see Table \ref{Table-simulations}. To check the accuracy of our simulations, we also calculated the estimated values for the mean of the perimeter and the area of the convex hull of the trajectory of planar Brownian motion run up to time $1$, even though these values are known explicitly. Comparing results obtained using simulations with the true values, we could see that the error exhibited by the simulations is of the order $10^{-3}$. Moreover, according to our simulations, we estimate the mean of the diameter of the convex hull of the trajectory of planar Brownian motion run up to time $1$ to be approximately $1.99$. This corroborates the estimate obtained by simulation in \cite{McRedmond_Xu}.

\subsection{Numerical evaluation of the second moment of the perimeter}\label{subsec:double_integral}
As already mentioned in the introduction, the authors of \cite{Wade_Xu} gave the following integral expression for the second moment of the perimeter of the convex hull of the planar Brownian motion run up to time $1$,
\begin{equation*}
\mathbb{E}\left[\calP^2\right]=4\pi\int_{-\frac{\pi}{2}}^\frac{\pi}{2}\int_0^\infty \cos\theta\frac{\cosh(u\theta)}{\sinh(u\pi/2)}\tanh\left(\frac{2\theta+\pi}{4}u\right)\D{u}\D{\theta}.
\end{equation*}
There are two problematic issues concerning the integral above, namely dividing by zero when $u = 0$, and overflow in $\cosh$ and $\sinh$ when $u$ becomes big. It turns out that both of these problems can be removed with the aid of  some simple symbolic manipulations. To solve the problem that appears at infinity, one can rewrite $\cosh$ and $\sinh$ using the definition, and then divide by the fastest growing term. This gives us
\begin{align*}
    \frac{\cosh(u\theta)}{\sinh(u\pi/2)}
    & = \frac{e^{u \theta} + e^{-u \theta}}{e^{u\pi/2} - e^{-u\pi/2}} = \frac{e^{u(\theta - \pi/2)} + e^{-u(\theta + \pi/2)}}{1 - e^{-u\pi}}.
\end{align*}
Furthermore, to get rid of the problem at $u=0$, we use the fact that $\tanh(x) = \frac{1 - e^{-2x}}{1+e^{-2x}}$, which results in
\begin{align*}
    \frac{\cosh(u\theta)}{\sinh(u\pi/2)} \tanh\left(\frac{2\theta+\pi}{4}u\right)
    & = \frac{e^{u(\theta - \pi/2)} + e^{-u(\theta + \pi/2)}}{1 - e^{-u\pi}} \cdot \frac{1 - e^{-u \theta - u\pi/2}}{1 + e^{-u \theta - u\pi/2}} \\
    & = \frac{e^{u(\theta - \pi/2)} + e^{-u(\theta + \pi/2)}}{1 + e^{-u \theta - u\pi/2}} \cdot \frac{1 - e^{-u \theta - u\pi/2}}{1 - e^{-u\pi}}.
\end{align*}
We observe that the function
\begin{equation*}
    u \mapsto \frac{1 - e^{-u \theta - u\pi/2}}{1 - e^{-u\pi}}
\end{equation*}
is no longer problematic at $u = 0$ as it holds that
\begin{equation*}
    \lim_{u \to 0} \frac{1 - e^{-u \theta - u\pi/2}}{1 - e^{-u\pi}} = \frac{1}{2} + \frac{\theta}{\pi}.
\end{equation*}
After making the modifications described above, \emph{Mathematica} was able to numerically evaluate the integral without issue. More specifically, we used the command \texttt{NIntegrate} with the following options: \texttt{WorkingPrecision\textrightarrow 200}, \texttt{Method\textrightarrow ``GaussKronrodRule''}, \texttt{Exclusions\textrightarrow\{\texttheta==-\textpi/2,\texttheta==\textpi/2\}}, and \texttt{AccuracyGoal\textrightarrow 10}. We were also able to reproduce the same result in \emph{MATLAB} using the command \texttt{integral2} with the following options: \texttt{`Method', `iterated', `AbsTol', 1e-25, `RelTol', 1e-10}. As reported in \eqref{eq:perimeter_2nd}, the numerical value of the integral (truncated at $8$ digits) is $26.209056$.

\section*{Acknowledgments} 
\noindent
Our collaboration started at the 10th International Conference on L\'{e}vy Processes that took place in Mannheim (Germany) in the summer of 2022 and we wish to thank the organizers for having created an excellent working atmosphere.
We are grateful to Kilian Raschel (Univ.\ Angers) for sharing with us manuscript \cite{Garbit-Raschel} and for stimulating communication. We also thank Andrew Wade (Univ.\ Durham) for reporting to us his thoughts on the numerical approximation of the integral from \eqref{eq:perimeter_2nd} and for fruitful discussions.
We want to express our gratitude to Andro Mer\'{c}ep (Univ.\ Zagreb) for his help by writing the code for the simulations. Not only did he optimize our initial ideas for the simulations, but he had a lot of patience for modifications we wanted to make along the way.
We thank Mateusz Kwa\'{s}nicki (Wroc\l{}aw University of Science and Technology) and Andrey Pilipenko (Institute of Mathematics of Ukrainian NAS) for recent useful discussions that led to further improvements of some of our results.
We also thank anonymous referees for their careful reading and thoughtful comments. Financial support through the \emph{Croatian Science Foundation} under project IP-2022-10-2277 (for S. \v{S}ebek) is gratefully acknowledged.

\bibliographystyle{abbrv}
\bibliography{hull_bib}

\begin{thebibliography}{10}

\bibitem{Vysotsky}
A.~Akopyan and V.~Vysotsky.
\newblock Large deviations of convex hulls of planar random walks and
  {B}rownian motions.
\newblock {\em Ann. H. Lebesgue}, 4:1163--1201, 2021.

\bibitem{Baxter}
G.~Baxter.
\newblock A combinatorial lemma for complex numbers.
\newblock {\em Ann. Math. Stat.}, 32:901--904, 1961.

\bibitem{theta_laws}
P.~Biane, J.~Pitman, and M.~Yor.
\newblock Probability laws related to the {J}acobi theta and {R}iemann zeta
  functions, and {B}rownian excursions.
\newblock {\em Bull. Amer. Math. Soc. (N.S.)}, 38(4):435--465, 2001.

\bibitem{Boyd}
S.~Boyd and L.~Vandenberghe.
\newblock {\em Convex optimization}.
\newblock Cambridge University Press, Cambridge, 2004.

\bibitem{absolute_moments}
B.~M. Brown.
\newblock Formulae for absolute moments.
\newblock {\em J. Austral. Math. Soc.}, 13:104--106, 1972.

\bibitem{Burkholder}
D.~L. Burkholder.
\newblock Exit times of {B}rownian motion, harmonic majorization, and {H}ardy
  spaces.
\newblock {\em Advances in Math.}, 26(2):182--205, 1977.

\bibitem{Cranston}
M.~Cranston, P.~Hsu, and P.~March.
\newblock Smoothness of the convex hull of planar {B}rownian motion.
\newblock {\em Ann. Probab.}, 17(1):144--150, 1989.

\bibitem{Cygan-Sandric-Sebek}
W.~Cygan, N.~Sandri{\'c}, and S.~{\v{S}}ebek.
\newblock Convex hulls of stable random walks.
\newblock {\em Electron. J. Probab.}, 27:30, 2022.
\newblock Id/No 98.

\bibitem{ElBachir}
M.~El~Bachir.
\newblock {\em L'Enveloppe convexe du mouvement brownien}.
\newblock PhD thesis, Université Toulouse III - Paul Sabatier, 1983.

\bibitem{Eldan}
R.~Eldan.
\newblock Volumetric properties of the convex hull of an {{\(n\)}}-dimensional
  {Brownian} motion.
\newblock {\em Electron. J. Probab.}, 19:34, 2014.
\newblock Id/No 45.

\bibitem{Feller}
W.~Feller.
\newblock The asymptotic distribution of the range of sums of independent
  random variables.
\newblock {\em Ann. Math. Statistics}, 22:427--432, 1951.

\bibitem{Garbit-Raschel}
R.~Garbit and K.~Raschel.
\newblock An improved lower bound for the expected diameter of planar brownian
  motion.
\newblock unpublished manuscript.

\bibitem{Gradshteyn_Ryzhik}
I.~S. Gradshteyn and I.~M. Ryzhik.
\newblock {\em Table of integrals, series, and products}.
\newblock Elsevier/Academic Press, Amsterdam, eighth edition, 2015.
\newblock Translated from the Russian, Translation edited and with a preface by
  Daniel Zwillinger and Victor Moll, Revised from the seventh edition
  [MR2360010].

\bibitem{Imhof}
J.-P. Imhof.
\newblock On the range of {B}rownian motion and its inverse process.
\newblock {\em Ann. Probab.}, 13(3):1011--1017, 1985.

\bibitem{Jovalekic}
M.~Jovaleki\'{c}.
\newblock Lower bound for the diameter of planar {B}rownian motion.
\newblock {\em Bull. Math. Soc. Sci. Math. Roumanie (N.S.)},
  64(112)(3):281--284, 2021.

\bibitem{Kabluchko-Zapor-TAMS}
Z.~Kabluchko and D.~Zaporozhets.
\newblock Intrinsic volumes of {Sobolev} balls with applications to {Brownian}
  convex hulls.
\newblock {\em Trans. Am. Math. Soc.}, 368(12):8873--8899, 2016.

\bibitem{Letac}
G.~Letac and L.~Tak\'{a}cs.
\newblock Problems and {S}olutions: {S}olutions of {A}dvanced {P}roblems: 6230.
\newblock {\em Amer. Math. Monthly}, 87(2):142, 1980.

\bibitem{Levy}
P.~L\'{e}vy.
\newblock {\em Processus {S}tochastiques et {M}ouvement {B}rownien. {S}uivi
  d'une note de {M}. {L}o\`{e}ve.}
\newblock Gauthier-Villars, Paris, 1948.

\bibitem{Majumdar}
S.~N. Majumdar, A.~Comtet, and J.~Randon-Furling.
\newblock Random convex hulls and extreme value statistics.
\newblock {\em J. Stat. Phys.}, 138(6):955--1009, 2010.

\bibitem{McRedmond_Xu}
J.~McRedmond and C.~Xu.
\newblock On the expected diameter of planar {B}rownian motion.
\newblock {\em Statist. Probab. Lett.}, 130:1--4, 2017.

\bibitem{Molchanov}
I.~Molchanov.
\newblock Convex and star-shaped sets associated with multivariate stable
  distributions. {I}: {Moments} and densities.
\newblock {\em J. Multivariate Anal.}, 100(10):2195--2213, 2009.

\bibitem{Molchanov-Wespi}
I.~Molchanov and F.~Wespi.
\newblock Convex hulls of {L{\'e}vy} processes.
\newblock {\em Electron. Commun. Probab.}, 21:11, 2016.
\newblock Id/No 69.

\bibitem{Morters_Peres}
P.~M\"{o}rters and Y.~Peres.
\newblock {\em Brownian motion}, volume~30 of {\em Cambridge Series in
  Statistical and Probabilistic Mathematics}.
\newblock Cambridge University Press, Cambridge, 2010.
\newblock With an appendix by Oded Schramm and Wendelin Werner.

\bibitem{Bonnesen}
R.~Osserman.
\newblock Bonnesen-style isoperimetric inequalities.
\newblock {\em Amer. Math. Monthly}, 86(1):1--29, 1979.

\bibitem{high_dim_hulls}
H.~Panzo and E.~Socher.
\newblock Bounds on some geometric functionals of high dimensional {B}rownian
  convex hulls and their inverse processes.
\newblock arXiv:2407.08712, 2024.

\bibitem{max_Bessel}
J.~Pitman and M.~Yor.
\newblock The law of the maximum of a {B}essel bridge.
\newblock {\em Electron. J. Probab.}, 4:no. 15, 35, 1999.

\bibitem{Revuz_Yor}
D.~Revuz and M.~Yor.
\newblock {\em Continuous martingales and {B}rownian motion}, volume 293 of
  {\em Grundlehren der Mathematischen Wissenschaften [Fundamental Principles of
  Mathematical Sciences]}.
\newblock Springer-Verlag, Berlin, third edition, 1999.

\bibitem{Rogers_Shepp}
L.~C.~G. Rogers and L.~Shepp.
\newblock The correlation of the maxima of correlated {B}rownian motions.
\newblock {\em J. Appl. Probab.}, 43(3):880--883, 2006.

\bibitem{harmonic_survey}
M.~A. Snipes and L.~A. Ward.
\newblock Harmonic measure distributions of planar domains: a survey.
\newblock {\em J. Anal.}, 24(2):293--330, 2016.

\bibitem{Snyder-Steele}
T.~L. Snyder and J.~M. Steele.
\newblock Convex hulls of random walks.
\newblock {\em Proc. Am. Math. Soc.}, 117(4):1165--1173, 1993.

\bibitem{Spitzer}
F.~Spitzer.
\newblock Some theorems concerning {$2$}-dimensional {B}rownian motion.
\newblock {\em Trans. Amer. Math. Soc.}, 87:187--197, 1958.

\bibitem{Spitzer-Widom}
F.~Spitzer and H.~Widom.
\newblock The circumference of a convex polygon.
\newblock {\em Proc. Am. Math. Soc.}, 12:506--509, 1961.

\bibitem{Cauchy}
E.~Tsukerman and E.~Veomett.
\newblock Brunn-{M}inkowski theory and {C}auchy's surface area formula.
\newblock {\em Amer. Math. Monthly}, 124(10):922--929, 2017.

\bibitem{Python}
G.~Van~Rossum and F.~L. Drake.
\newblock {\em Python 3 Reference Manual}.
\newblock CreateSpace, Scotts Valley, CA, 2009.

\bibitem{ch_bm_bb}
S.~\v{S}ebek.
\newblock Convex hull of brownian motion and brownian bridge.
\newblock arXiv:2406.07079, 2024.

\bibitem{Vysotsky-Zapor}
V.~Vysotsky and D.~Zaporozhets.
\newblock Convex hulls of multidimensional random walks.
\newblock {\em Trans. Am. Math. Soc.}, 370(11):7985--8012, 2018.

\bibitem{Wade-Xu-SPA-2015}
A.~R. Wade and C.~Xu.
\newblock Convex hulls of random walks and their scaling limits.
\newblock {\em Stochastic Processes Appl.}, 125(11):4300--4320, 2015.

\bibitem{Wade_Xu}
A.~R. Wade and C.~Xu.
\newblock Convex hulls of random walks and their scaling limits.
\newblock {\em Stochastic Process. Appl.}, 125(11):4300--4320, 2015.

\bibitem{welzl}
E.~Welzl.
\newblock Smallest enclosing disks (balls and ellipsoids).
\newblock In {\em New Results and New Trends in Computer Science: Graz,
  Austria, June 20--21, 1991 Proceedings}, pages 359--370. Springer, 2005.

\bibitem{Widder}
D.~V. Widder.
\newblock {\em The {L}aplace {T}ransform}, volume vol. 6 of {\em Princeton
  Mathematical Series}.
\newblock Princeton University Press, Princeton, NJ, 1941.

\end{thebibliography}

\end{document}